\documentclass[11pt]{amsart}

\newtheorem{theorem}{Theorem}
\newtheorem{lemma}{Lemma}
\newtheorem{proposition}{Proposition}
\theoremstyle{remark}
\newtheorem{remark}{Remark}

\newcommand\R{\mathbb{R}}

\newcommand\ud{\,\textnormal{d}}
\newcommand\uds{\textnormal{d}}
\newcommand\AND{\quad\mbox{and}\quad}
\newcommand\ONE{\mathbf{1}}

\newcommand\ii{\textnormal{i}}
\newcommand\ee{\textnormal{e}}

\newcommand\px{\partial_x}
\newcommand\pt{\partial_t}

\newcommand\oo{\omega_0}
\newcommand\po{\phi_\omega}
\newcommand\poo{\phi_{\oo}}
\newcommand\Lo{\Lambda_\omega}

\newcommand\LP{L_+}
\newcommand\LM{L_-}
\newcommand\MP{M_+}
\newcommand\MM{M_-}
\newcommand\QP{Q_+}
\newcommand\QM{Q_-}

\newcommand\cC{\mathcal{C}}
\newcommand\cI{\mathcal{I}}
\newcommand\cJ{\mathcal{J}}
\newcommand\cK{\mathcal{K}}

\newcommand\Xe{X_\alpha}
\newcommand\Ye{Y_\alpha}

\DeclareMathOperator\sech{sech}
\DeclareMathOperator\sgn{sgn}

\title[Asymptotic stability of solitons for cubic-quintic NLS]
{Asymptotic stability of solitary waves for the 1D cubic-quintic Schrödinger equation with no internal mode}
\author{Yvan Martel}
\address{CMLS, \'Ecole polytechnique, CNRS, Institut Polytechnique de Paris, 91128 Palaiseau Cedex, France}
\email{yvan.martel@polytechnique.edu}

\begin{document}
\begin{abstract}
For the Schrödinger equation with a cubic-quintic, focusing-defocusing nonlinearity in one space dimension, we prove 
the asymptotic stability of solitary waves for a large range of admissible frequencies.
For this model, the linearized problem around the solitary waves does not have internal mode nor resonance.
\end{abstract}
\maketitle

\section{Introduction}

In this article, we consider the one-dimensional Schrödinger equation with focusing cubic and defocusing quintic nonlinearities
\begin{equation}\label{NLS}
\ii \pt \psi + \px^2 \psi + |\psi|^2 \psi - |\psi|^4 \psi = 0, \quad (t,x)\in\R\times\R.
\end{equation}
The corresponding Cauchy problem is globally well-posed in the energy space~$H^1(\R)$ (see \emph{e.g.}~\cite{Cbook}) and this model enjoys the conservation laws
\begin{align*}
\mathcal M [\psi] &= \int |\psi|^2 \ud x \tag{Mass}\\
\mathcal P [\psi] &= \Im \int \psi \px \bar \psi \ud x \tag{Momentum}\\
\mathcal E [\psi] &= \int \biggl( \frac 12 | \px \psi|^2 - \frac 14|\psi|^4 + \frac 16 |\psi|^6 \biggr)\uds x. \tag{Energy}
\end{align*}
We recall the Galilean transform, translation and phase invariances of~\eqref{NLS}: if $\psi(t,x)$ is a solution then, for any $\beta,\sigma,\gamma\in \R$, the function
\begin{equation}\label{eq:inv}
\widetilde \psi(t,x)= \ee^{\ii (\beta x-\beta^2 t+\gamma)}\psi(t,x-2 \beta t-\sigma)
\end{equation}
is also a solution.

\smallskip

It follows from well-known arguments (see \emph{e.g.}~\cite{BL}) that for any~$\omega\in (0,\frac 3{16})$, there exists a unique positive even solution~$\po\in H^1(\R)$ of
the equation
\begin{equation}\label{eq:po}
\po'' + \po^3 - \po^5 = \omega \po,\quad x\in \R,
\end{equation}
whereas for $\omega \geq \frac {16}3$, there exists no solution of \eqref{eq:po} in $H^1(\R)\setminus \{0\}$.
Moreover, for $\omega\in (0,\frac 3{16})$, the solution is actually explicit (see \emph{e.g.} ~\cite{CS},~\cite{Oh}, \cite{PKA} and~\cite[Chapter 5]{Ybook})
\[
\po (x) = \sqrt{\frac {4\omega}{1+a_\omega \cosh(2\sqrt{\omega}x)}}\quad\mbox{where}\quad
a_\omega = \sqrt{1-\frac{16}3\omega}.
\]
For any~$\omega\in (0,\frac 3{16})$, the function~$\psi(t,x)=\ee^{\ii\omega t}\po(x)$ is a standing wave solution of~\eqref{NLS}.
The invariances~\eqref{eq:inv} generate a family of traveling waves, of the form
\[
\psi(t,x)=\ee^{\ii(\beta x-\beta^2 t + \omega t +\gamma)} \po(x-2 \beta t-\sigma)
\]
for $\beta,\sigma,\gamma\in \R$.
The stability of these solutions by perturbation of the initial data in the energy space $H^1(\R)$ 
is a classical question.
We recall the orbital stability result for standing waves from~\cite{Oh} and we refer to~\cite{CL,GSS,IK,We85,We86,Ybook} for previous related works.

\begin{proposition}[{\cite[Theorem~3]{Oh}}]\label{PR:1}
For any~$\omega_0 \in (0,\frac 3{16})$ and any~$\varepsilon>0$, there exists~$\delta>0$
 with the following property: if~$\psi_0\in H^1(\R)$,~$\| \psi_0 - \poo\|_{H^1(\R)} < \delta$
and~$\psi$ is the solution of~\eqref{NLS} with~$\psi(0)=\psi_0$, then
\begin{equation}\label{eq:S}
\sup_{t\in \R}\inf_{(\gamma,\sigma)\in \R^2}\|\psi(t,\cdot+\sigma) - \ee^{\ii \gamma} \poo\|_{H^1(\R)}
< \varepsilon.
\end{equation}
\end{proposition}

To complement this stability result, 
we prove the asymptotic stability of a large range of standing waves of~\eqref{NLS}
by perturbations in the energy space.

\begin{theorem}\label{TH:1}
For any~$\omega_0 \in (0,\frac 18]$, there exists~$\delta>0$ with the following property: if~$\psi_0\in H^1(\R)$,~$\| \psi_0 - \poo\|_{H^1(\R)} < \delta$ and~$\psi$ is the solution of~\eqref{NLS} with~$\psi(0)=\psi_0$, then there exist~$\beta_+\in\R$ and $\omega_+\in (0,\frac 3{16})$ such that
for any bounded interval~$I$ of~$\R$
\begin{equation}\label{eq:AS}
\lim_{t\to +\infty} \inf_{(\gamma,\sigma)\in \R^2} \sup_{x\in I}
\big|\psi(t,x+\sigma) - \ee^{\ii \gamma} \ee^{\ii \beta_+ x}\phi_{\omega_+}(x)\big|=0.
\end{equation}
\end{theorem}

\begin{remark}
The orbital stability property \eqref{eq:S} means that the solution stays for all time close to the family of solitary waves and more precisely close to the initial solitary wave $\poo$, up to phase and translation. In particular, in~\eqref{eq:S} one can replace $\poo$ by $\ee^{\ii \beta x} \phi_{\omega}$ for 
$\beta$ small and $\omega$ close to $\omega_0$.
In contrast, the asymptotic stability property~\eqref{eq:AS} says that as $t\to +\infty$, the solution converges locally in space to a final asymptotic soliton, characterized by $\beta_+$ and $\omega_+$, up to phase and translation. As a consequence, the values of $\beta_+$ and $\omega_+$ such that \eqref{eq:AS} holds are unique
and depend in an intricate way of the initial data.
By the time reversibility of equation \eqref{NLS}, the same holds as $t\to -\infty$, with possibly different parameters $\beta_-$ and $\omega_-$.
From the orbital stability, it follows in the context of Theorem~\ref{TH:1} that~$|\beta_\pm|$ and~$|\omega_\pm-\omega_0|$ are arbitrarily small for 
$\| \psi_0 - \poo\|_{H^1(\R)}$ small.

The asymptotic stability property~\eqref{eq:AS} is stated in the $\sup$ norm on any compact interval for the sake of simplicity.
The proof provides additional information on the asymptotic behavior of the solution: there exist $\cC^1$ time dependent functions $(\beta,\sigma,\gamma,\omega) :[0,\infty)\to \R^3\times(0,\frac 3{16})$ with
 $\lim_{\infty} \beta=\beta_+$, $\lim_{\infty} \omega=\omega_+$, such that
the function $u$ defined by
\begin{equation}\label{def:uintro}
u(t,x)=\ee^{-\ii \gamma(t)} \ee^{-\ii \beta(t) x } \psi(t,x+\sigma(t)) - \phi_{\omega(t)}(x)
\end{equation}
satisfies for some constant $c_0>0$,
\begin{equation*}
\int_0^\infty \int_\R \ee^{-c_0 |x|} \left( |\px u(t,x)|^2 + |u(t,x)|^2\right) \uds x
\ud t <\infty.
\end{equation*}
Therefore, there exists a sequence $t_n\to +\infty$ such that for any compact interval~$I$, $\lim_{n\to+\infty} \int_I |\px u(t_n,x)|^2\ud x=0$. The convergence as $t\to +\infty$ is conjectured, but we do not pursue this issue here. Recall that in general the global norm $\|u(t)\|_{H^1(\R)}$
does not converge to zero since by the stability result and the time reversibility of the equation, this would imply that $\psi$ is exactly a solitary wave.
From the proof, one also shows that $\lim_{\infty} \dot \sigma = 2 \beta_+$
and $\lim_{\infty} \dot \gamma = \beta_+^2+\omega_+$.

For $H^1(\R)$ perturbations, it is unlikely that more can be said in general about the decay rate of $u(t)$ and the asymptotic behaviors of $\sigma$ and $\gamma$.
We refer to~\cite[Theorem~1.3]{GNT2} for the lack of decay rate for the perturbation part
and to~\cite[Theorem~2]{MM2} for a discussion on this issue for the Korteweg-de Vries equation.
\end{remark}

\begin{remark}
By the Galilean transform, translation and phase invariances, the family of traveling solitary waves enjoys the same stability properties. In particular, for $(\beta_0,\sigma_0,\gamma_0,\omega_0)\in \R^3\times (0,\frac 18]$, the result of Theorem~\ref{TH:1} holds for initial data sufficiently close to 
$\ee^{\ii\gamma_0}\ee^{\ii\beta_0 x}\poo(x-\sigma_0)$ for any $\beta_0,\sigma_0,\gamma_0\in \R$.
\end{remark}

\begin{remark}
It would simplify the proof to show Theorem~\ref{TH:1} only for small values of $\omega_0$, instead of considering the
explicit range $(0,\frac 18]$.
The value~$\frac 18$ is chosen for its simplicity and it is not sharp in our approach. 
However, a proof for the full range~$(0,\frac {3}{16})$ would certainly require additional arguments because of the specific behavior in the limit~$\omega_0\uparrow \frac3{16}$.
\end{remark}

For the integrable one-dimensional focusing cubic Schr\"odinger equation
\begin{equation}\label{INLS}
\ii \pt \psi + \px^2 \psi + |\psi|^2 \psi= 0,
\end{equation}
the Inverse Scattering Transform theory was successfully applied in~\cite{CP} to prove the asymptotic stability of solitons in~$L^2$ weighted spaces. 
See also~\cite{BJL} and the references in these papers concerning the IST theory.
Note that the asymptotic stability of solitary waves for initial perturbations in the energy space~$H^1(\R)$ as stated in Theorem~\ref{TH:1} is not true in the integrable case.
Indeed, counterexamples are provided by the family of multi-solitons constructed in~\cite[\S 5]{ZS}. More explicitly, formula (2.15) page 333 of~\cite{Ol} with the choice of parameters $b_1=b_2=1$,~$\xi_1=\xi_2=0$,~$\eta_1=\frac 12$,~$\eta_2=\frac \eta2\ll 1$ gives an explicit periodic solution of~\eqref{INLS} with a two soliton structure, which is arbitrarily close in~$H^1(\R)$ to the soliton $\sqrt{2}\sech(x)$ when~$\eta$, parameter related to the size of the small soliton, is small.
In the limit $\omega_0$ small, Theorem~\ref{TH:1} shows that asymptotic stability of solitary waves in the energy space holds for models perturbative of~\eqref{INLS}.
This observation is related to results of asymptotic stability of kinks proved for wave-type models close to the sine-Gordon equation in \cite[Theorems~8 and 9]{KMMV}.

\smallskip

We refer to~\cite{DZ0,DZ} for a description of the long time behavior of solutions of the defocusing cubic Schr\"odinger equation and some of its perturbations.
Several articles are concerned with the one-dimensional cubic Schr\"odinger equation with a potential, see the most recent ones~\cite{GP,G,CM2,GNT2,Mi,Na} and their references.
More generally, there is a vast literature about the asymptotic stability of waves for nonlinear Schr\"odinger equations, with or without potential, for any space dimension and various nonlinearities; see the surveys~\cite{CM,KMMls,Sc07}.
We refer to~\cite{BP1,BP2,CGNT,CG,KS,KNS} for results directly related to the one-dimensional case with no potential.

\smallskip

The presence of internal modes is known to greatly complicate the analysis of asymptotic stability by shifting the problem to the nonlinear level, where a condition, called the Fermi golden rule, then enters into play (see \emph{e.g.} the review~\cite{CM}). 
Roughly, internal modes (as defined in~\cite{CG,CM3,PKA}) generate linear periodic solutions to the linearized evolution equation around the solitary wave. 
 For the Schr\"odinger equation with power nonlinearity~$|u|^{p-1}u$ with~$p\neq 3$ close to~$3$, the existence of internal modes bifurcating from the resonance for~$p=3$ is proved in~\cite{CG} (see~\cite{CGNT} for the case of~$p$ close to~$5$).
For the Schr\"odinger equation~$\ii \pt \psi + \px^2 \psi + |\psi|^2 \psi+|\psi|^4\psi = 0$, with focusing cubic-quintic nonlinearity, the existence of internal modes is shown in~\cite{PKA}.
It is also shown there that for equation~\eqref{NLS}, there does not exist internal mode.
This observation and the proximity of model~\eqref{NLS} to the integrable case were the original motivations to study it in the present paper. Concerning its relevance in Physics, the reader may consult for instance~\cite{CS}, \cite[p.769]{KM}, \cite{Oh,PKA,SP}, \cite[Chapter~5]{Ybook}.

\smallskip

Our approach to prove Theorem~\ref{TH:1} is directly inspired by~\cite{KMM,KMMV} (see also~\cite{KMM1}) proving the asymptotic stability of solitons and kinks for one-dimensional wave-type models, in the absence of resonance and internal mode,
by using virial estimates and a transformed problem. We also point out that algebraic facts on the linearized operator around a solitary wave, established in~\cite{CGNT} (see Lemma~2 in the present paper) are decisive in this approach.
We refer to \S \ref{s:2.3} and \S \ref{s:3.2} for more comments on the proof.
For the one-dimensional cubic Schr\"odinger equation with a real potential, the article~\cite{CM2} uses a similar approach, suitably combined with the notion of refined profiles introduced in~\cite{CM3,CM}, to deal with the presence of more than one discrete mode for the potential.

\smallskip

Finally, we point out that similar virial estimates were used extensively to prove the asymptotic stability of solitons of the subcritical generalized Korteweg-de Vries equations (see~\cite{M,MM2} and references therein), but also to study the singularity formation for the mass critical generalized Korteweg-de Vries equation and nonlinear Schr\"odinger equation (see~\cite{MM,MMR,MR}). Indeed, a sharp description of the bubbling phenomenon for these equations requires sharp estimates on the error term out of reach of standard energy methods.

\subsection*{Notation}
The letters~$u$,~$v$,~$w$ and~$z$ denote complex-valued functions with \emph{e.g.},~$u=u_1+\ii u_2$, $u_1,u_2\in \R$.
The letters~$g$ and $h$ denote real-valued functions.
The Fourier transform of a function~$w$ is denoted by~$\hat w$.
For~$\alpha>0$, set
\[
\Xe = (1-\alpha \px^2)^{-1} \quad
\mbox{\emph{i.e.}}\quad \widehat{\Xe w} (\xi) = \frac {\hat w(\xi)}{1+\alpha \xi^2}\quad \mbox{for~$\xi\in \R$
and $w\in L^2(\R)$.}
\]
Denote~$\langle u , v \rangle = \Re \int u\bar v \ud x$ and~$\|u\|=\sqrt{\langle u , u \rangle}$.
Last, we denote
\[
f(u) = |u|^2u - |u|^4 u,\quad F(u) = \frac{|u|^4}{4} - \frac{|u|^6}6.
\]
In this paper, $C$ denotes various positive constants which do not depend on the parameters
$\oo$, $\varepsilon$, $\alpha$, $A$ and $B$, except when such parameters (like $B$ and~$\alpha$
at the end of the proof of Proposition~\ref{PR:4}) are eventually fixed.

\subsection*{Acknowledgments}
The author would like to thank Thierry Cazenave, Scipio Cuccagna, Philippe Gravejat and Pierre Rapha\"el
for early discussions on the asymptotic stability problem for nonlinear Schr\"odinger equations.
The collaboration with {Micha\l} Kowalczyk, Claudio Mu\~noz and Hanne Van Den Bosch on wave-type models has been decisive for this project.

\section{Preliminaries}

\subsection{Solitary waves}

We gather basic properties of~$\po$ and~$\Lo=\omega \frac{\partial \po}{\partial \omega}$.
\begin{lemma}\label{LE:1}
For any $k\geq 0$, there exists $C_k>0$ such that for any $\omega\in (0,\frac 18]$ and any $x\in \R$,
\begin{equation*}
|\po^{(k)}(x)|\leq C_k \, \omega^\frac{1+k}2 \ee^{-\sqrt{\omega}|x|},\quad
|\Lo^{(k)}(x)|\leq C_k\, \omega^\frac{1+k}2 \left(1+\sqrt{\omega}|x|\right) \ee^{-\sqrt{\omega}|x|}.
\end{equation*}
Moreover, there exists $c>0$ such that $\langle \po,\Lo\rangle \geq c \sqrt{\omega}$.
\end{lemma}
\begin{proof}
The bounds follow from the explicit expressions of~$\po$ and $\Lo$. 
The restriction $\omega\in(0,\frac 18]$ implies $a_\omega \geq 1/\sqrt{3}$ and so it allows us to obtain constants $C_k$ independent of $\omega$.

By change of variable, we compute
$\|\po\|^2 = 2 \sqrt{\omega} \int (1+a_\omega \cosh(y))^{-1}\ud y$.
The lower bound on $\langle \po,\Lo\rangle=\frac 12 \omega \frac{\partial }{\partial \omega}\|\po\|^2$
follows from the observation that the map~$\omega\in (0,\frac{3}{16})\mapsto a_\omega$ is decreasing.
The fact that~$\langle \po,\Lo\rangle >0$ is related to orbital stability (\cite{GSS,IK,Oh,We86}).
\end{proof}

\subsection{Spectral properties}\label{S:2.2}
The linearization of~\eqref{NLS} around~$\po$ involves the operators
(see \emph{e.g.} \cite{CGNT,We85})
\[
 \LP = - \px^2 + \omega - 3 \po^2 + 5 \po^4
\AND
 \LM = - \px^2 + \omega - \po^2 + \po^4.
\]
We recall a few properties of the operators $\LP$ and $\LM$ and refer to \cite{We85,We86} and \cite[Lemma 2.2]{CGNT} for details.
The operator~$ \LP$ has exactly one negative eigenvalue. Moreover, the kernel of~$ \LP$ is generated by~$\po'$.
Differentiating the equation of~$\po$ with respect to~$\omega$, we obtain~$ \LP \Lo = -\omega\po$.
Last,~$ \LM\geq 0$ and its kernel is generated by~$\po$. 

\subsection{Conjugate identity}\label{s:2.3}

We adapt to the present context an identity from~\cite[\S 3.4]{CGNT}.
Let
\begin{gather*}
S = \po \cdot \px \cdot \frac 1 {\po},\quad
S^* = - \frac 1 {\po} \cdot \px \cdot \po,\\
\MP = - \px^2 + \omega-\frac13 \po^{4} ,\quad \MM = - \px^2 + \omega +\po^{4}.
\end{gather*}
The above definitions of $S$ and $S^*$ mean that for a function $g$, $Sg = \po \big( \frac{g}{\po}\big)'$
and $S^* g = - \frac 1\po (g\po)'$.

\begin{lemma}\label{LE:2}
For any~$\omega\in (0,\frac 3{16})$,~$S^2 \LP \LM = \MP \MM S^2$.
\end{lemma}

The identity proved in Lemma~\ref{LE:2} is different from the ones involved in~\cite{CM2,KMM,KMMV} to define the transformed problem for the Schr\"odinger equation with a real-valued potential or for wave-type equations.
See also Remark~\ref{RK:3}.

The interest of this identity lies on the properties of the operators $\MP$ and $\MM$.
Indeed, the potential of $\MM$ is repulsive (in the sense that $x(\po^4)'\leq 0$ on $\R$)
while the potentials of $\LP$ and $\LM$ are not repulsive.
Recall that the repulsive nature of a potential is decisive to apply a virial argument (see~\cite[Theorem~XIII.60]{RS}).
The potential of $\MP$ is not repulsive, but being in absolute value three times less than the one of $\MM$, it is possible to control it for a large range of values of $\omega$ (see the proof of Proposition~\ref{PR:4} and in particular, the definition of the functional $\cK$).

The intuition that $S^2$ can be factorized at the left of the operator $S^2\LP\LM$
comes from the relations
\begin{gather}
\LM \po=0,\quad \LP\LM(x\po)=\LP (x\LM\po - 2 \po')=-2 \LP \po'=0,\nonumber\\
S\po=0 \AND S^2(x\po)=S\po=0. \label{eq:S2}
\end{gather}

The identity of Lemma~\ref{LE:2} is general and not specific to the nonlinearity in~\eqref{NLS}.
For example, in the case of the integrable equation~\eqref{INLS}, one obtains $\MP=\MM=-\px^2+\omega$.
This was a strong motivation to work in some sense close to the integrable case in the present paper.
We refer to similar observations on the sine-Gordon equation in~\cite{KMMV}.
Last, we point out that such approach by factorization seems limited to one-dimensional problems.

\begin{proof}[Proof of Lemma~\ref{LE:2}]
We recall from~(3.25)-(3.26) of~\cite{CGNT} (with a slight change of notation) 
the general formula
\[
( \px - R) ( \px^2 - V_+)( \px + R) = ( \px + R) ( \px^2 - V_-)( \px - R)
\]
where
\[
V_\pm = R^2 \pm 3 R' + \frac {R''}{R} .
\]
We set~$R = \frac{\po'}{\po}$. Using the identities
\begin{equation}\label{eq:id}
\po''=\omega \po - \po^3 + \po^5 
\AND (\po')^2 = \omega\po^2 - \frac 12 \po^4 + \frac 13 \po^6
\end{equation}
we compute
\begin{equation*}
R^2 = \omega - \frac 12 \po^2 + \frac 13 \po^4,\quad 
R' = -\frac 12 \po^2 + \frac 23 \po^4,\quad
\frac{R''}{R} = -\po^2 + \frac 83 \po^4,
\end{equation*}
and thus
\begin{equation*}
V_+ = \omega - 3 \po^2 + 5 \po^4,\quad
V_- = \omega + \po^4.
\end{equation*}
Observing that
\[
 \px - R = S,\quad \px + R = - S^*,\quad \LP = - \px^2 + V_+,\quad 
\MM = - \px^2 + V_-,
\] 
the general formula implies that~$S \LP S^* = S^* \MM S$.
Following~\cite{CGNT}, we also check that~$ \LM = S^* S$ and~$\MP = S S^*$ using the identities in~\eqref{eq:id}.
Thus, composing~$S \LP S^* = S^* \MM S$ by~$S$ on the left and on the right yields the identity~$S^2 \LP \LM = \MP\MM S^2$.
\end{proof}

\subsection{Invertibility of $\LP$ and $\MM$}\label{S:2.4}
We briefly discuss the invertibility of the operators $\LP$ and $\MM$.
Let~$G$ be the even solution of~$L_+ G=0$ such that~$\po'' G - \po' G' =1$ on~$\R$.
We check that~$|G^{(k)}(x)|\leq C_k \omega^{\frac{k-3}2}\ee^{\sqrt{\omega} |x|}$, for constants $C_k>0$.
For any bounded continuous function~$W$, define
\[
I_+ [W](x) = 
\begin{cases}
- \po'(x) \int_0^x G W - G(x) \int_x^\infty \po' W & \mbox{for~$x\geq 0$}\\
\po'(x) \int_x^0 G W + G(x) \int_{-\infty}^x \po' W & \mbox{for~$x< 0$}
\end{cases}
\]
Note that if~$\langle W,\po'\rangle=0$, then we have~$-\int_x^\infty \po' W=\int_{-\infty}^x \po' W$
so that the two expressions coincide at~$x=0$ and provide a solution to~$ \LP U = W$.

Last, denote by~$H_1$ and~$H_2$ two solutions of 
$\MM H_1=\MM H_2=0$ satisfying
\[
|H_1^{(k)}(x)| \leq C_k \omega^{-\frac 14+\frac k2}\ee^{-\sqrt{\omega} x},\quad
|H_2^{(k)}(x)| \leq C_k \omega^{-\frac 14+\frac k2}\ee^{\sqrt{\omega} x}
\]
for $C_k>0$ 
and $H_1 H_2' - H_1' H_2=1$ on $\R$.
Two such independent solutions exist because $\MM>0$ and so the equation~$\MM U=0$ does not have an~$H^1$ solution. For any bounded continuous function~$W$, define
\[
J_-[W](x) = H_1(x) \int_{-\infty}^x H_2 W + H_2(x) \int_x^{\infty} H_1 W.
\]
This formula defines a solution to~$\MM U =W$.

\section{Asymptotic stability}

\subsection{Modulation}
We fix $\oo \in (0,\frac 3{16})$ and an initial data
$\psi_0\in H^1(\R)$ close to $\poo$. By Proposition~\ref {PR:1}, the global solution~$\psi$ of \eqref{NLS} is close to the family of solitary waves
for all time.
It is standard to decompose~$\psi$ as
\begin{equation}\label{def:u}
\psi(t,y) = \ee^{\ii(\beta(t)(y-\sigma(t)) + \gamma(t))} \left[ \phi_{\omega(t)} (y-\sigma(t)) + u(t,y-\sigma(t))\right ] 
\end{equation}
(see also the equivalent formulation~\eqref{def:uintro} in the Introduction)
where the time-dependent functions~$\beta$,~$\sigma$,~$\gamma$ and~$\omega$ are of class~$\cC^1$ 
and uniquely fixed so that, for all $t\geq 0$, the following orthogonality relations hold
\begin{equation}\label{eq:ortho}
\langle u, \po\rangle = \langle u, x\po\rangle = \langle u, \ii \Lo\rangle =\langle u,\ii \po' \rangle=0. 
\end{equation}
This specific choice of orthogonality relations is known to yield quadratic estimates
for the time derivative of the modulation parameters, for all $t\geq 0$,
\begin{equation}\label{eq:modulation}
\frac{|\dot \beta |}{\sqrt{\omega}} + \frac{| \dot \omega|}\omega + \sqrt{\omega}|\dot\sigma-2\beta|
+ |\dot \gamma - \omega -\beta^2| \leq C\sqrt{\omega} \|\sech(\sqrt{\omega} x/2) u\|^2 .
\end{equation}
See \cite[(i) of Proposition 2.4]{We85} and \emph{e.g.} \cite[proof of Lemma~12]{Ng}.
Moreover, for~$\varepsilon>0$ small, we can formulate the orbital stability result of Proposition~\ref {PR:1} as follows, for all $t\geq 0$,
\begin{equation}\label{eq:small}
\|\px u\| + \|u\| + |\beta| + |\omega-\oo|\leq \varepsilon.
\end{equation}
Using the equation of~$\psi$, we check that $u=u_1+\ii u_2$ satisfies
\begin{equation}\label{eq:u}
\begin{cases}
 \pt u_1 = \LM u_2 + \theta_2 + m_2 - q_2\\
 \pt u_2 = - \LP u_1 - \theta_1 - m_1 + q_1
\end{cases}
\end{equation}
where
\begin{align*}
\theta_1 &= \dot \beta x \po + (\dot \gamma - \omega - \beta^2) \po\\
\theta_2 &= - \frac{\dot \omega}\omega \Lo + (\dot\sigma-2\beta) \po'\\
m_1 &= \dot \beta xu_1 + (\dot \gamma - \omega - \beta^2 ) u_1- (\dot \sigma -2\beta) \px u_2\\
m_2 &= \dot \beta xu_2 + (\dot \gamma - \omega - \beta^2 ) u_2 + (\dot \sigma -2\beta) \px u_1
\end{align*}
and
\begin{align*}
q_1 &= \Re \left\{ f(\po+u)-f(\po)-f'(\po)u_1 \right\} \\
q_2 &= \Im \left\{f(\po+u)-i\frac{f(\po)}{\po}u_2\right\}.
\end{align*}

\subsection{Outline of the proof of Theorem~\ref{TH:1}}\label{s:3.2}
The proof of Theorem~\ref{TH:1} relies on two localized virial estimates.
By localized, we mean that the functionals make sense for $H^1(\R)$ functions, 
by truncating the function $x$ involved in a virial computation at two different large spatial scales $A\gg B\gg 1$.

The first virial estimate is performed on the function $u$, solution of~\eqref{eq:u}.
Since the operators $\LP$ and $\LM$ are not repulsive, a virial computation is not sufficient to prove directly convergence to zero of $u$, but it allows to estimate it at any large spatial scale $A$ 
by a norm of $u$ with a weight $\rho$ related to the spatial decay of $\po$. See Proposition~\ref{PR:2}.

The second virial estimate is performed on the transformed function $v$, whose definition originates from the identity in Lemma~\ref{LE:2}. See \S \ref{S:transformed} for the definition of $v$ and Proposition~\ref{PR:3}.
The equation~\eqref{eq:v} of~$v$ involves the operators $\MP$ and $\MM$ which are better suited for a virial computation. This virial estimate in $v$ at the spatial scale $B$ contains error terms in $u$, in particular, nonlinear terms and modulation terms. In this step, we fix the constants $B$ and $\alpha$, independently of 
$\oo$, $\varepsilon$ and $A$.

Last, in Proposition~\ref{PR:4}, we estimate the function $u$ by the transformed function $v$ for suitable weighted norms, using the special orthogonality relations~\eqref{eq:ortho}.

Gathering the estimates of Propositions~\ref{PR:2}, \ref{PR:3} and \ref{PR:4}, adjusting the choice of~$A\gg1$ and taking $\varepsilon>0$ small enough, we complete the proof of Theorem~\ref{TH:1} in~\S \ref{s:3.5}.

\subsection{First virial estimate}
The following definitions are taken from~\cite{KMM}.
We fix a smooth even function~$\chi:\R\to \R$ satisfying
\[
\mbox{$\chi=1$ on~$[0,1]$,~$\chi=0$ on~$[2,+\infty)$,
$\chi'\leq 0$ on~$[0,+\infty)$.}
\]
For~$K>0$, define
\begin{gather*}
\chi_K(x)=\chi\left(\frac{\oo\sqrt{\oo}} K x\right),\quad 
\eta_K(x)=\sech\left(\frac{2\oo\sqrt{\oo}} K x\right), \\
\zeta_K(x)=\exp\left(-\frac{\oo\sqrt{\oo}}K|x|(1-\chi(\sqrt{\oo} x))\right), \quad
\Phi_K(x)=\int_0^x \zeta_K^2(y) \ud y.
\end{gather*}
Let~$1\ll B\ll A$ be large constants to be defined later.
Define
\[\Psi_{A,B}=\chi_A^2\Phi_B.\]
Technically, in the definitions of $\chi_K$, $\eta_K$ and $\zeta_K$, the multiplicative factor $ {\oo\sqrt{\oo}}/{K}$ will allow us to choose a parameter~$B$ large enough independent of
$\oo$ in the proof of Proposition~\ref{PR:3}.

Last, we introduce a weight function, related to the solitary wave $\po$
\[
\rho(x)=\sech\biggl( \frac{\sqrt{\oo}}{10} x\biggr).
\]
Since $|\omega-\oo|\leq \varepsilon$ by \eqref{eq:small}, for $\varepsilon<\frac{\oo}2$, we
have $\frac{\oo}2 \leq \omega\leq \frac32\oo$, which allows by Lemma~\ref{LE:1} to control $\po$, $\Lo$ and their derivatives in terms of powers of~$\rho$.

\begin{proposition}\label{PR:2}
There exists $C>0$ such that for $\varepsilon>0$ small enough, for any~$T\geq 0$,
\[
\int_0^T 
\left(\|\eta_A \px u \|^2+ \frac {\oo^3}{A^2}\|\eta_A u\|^2\right)\uds t
\leq C\varepsilon+ C\oo \int_0^T \|\rho^2 u\|^2\ud t .
\]
\end{proposition}
\begin{proof}
Define
\[
\cI= \int u_1 \left(2 \Phi_A \px u_2+ \Phi_A' u_2\right)
\AND
w=\zeta_A u.
\]
We claim, for all $t\geq 0$,
\begin{equation}\label{eq:dtI}
\dot{\cI} \geq \int |\px w|^2-C \oo\|\rho^2u\|^2.
\end{equation}
From the equation~\eqref{eq:u} of~$(u_1,u_2)$ and~$\int \left(2\Phi_A \px u_k+\Phi_A' u_k\right) u_k = 0$,
\begin{align*}
\dot \cI
& = 
- \sum_{k=1,2}
\int \left(2 \Phi_A \px u_k+ \Phi_A' u_k\right) \px^2 u_k\\
&\quad - \Re \left\{ \int \left(2 \Phi_A \px \bar u + \Phi_A' \bar u \right) \left(f(\po+u)-f(\po) \right)\right\} \\
&\quad +\sum_{k=1,2} \int \left(2 \Phi_A \px u_k+ \Phi_A' u_k\right) (\theta_k+m_k).
\end{align*}
Integrating by parts (see \emph{e.g.}~\cite[Lemma 1]{KMM}), for~$k=1,2$, we compute
\[
- \int \left(2 \Phi_A \px u_k+ \Phi_A' u_k\right) \px^2u_k 
= 2 \int ( \px w_k)^2 + \int (\ln \zeta_A)'' w_k^2
\]
where
\[
(\ln \zeta_A)''
=\frac{\oo^2}{A} \left( \sqrt{\oo} |x|\chi''(\sqrt{\oo} x)
+ 2 \chi'(\sqrt{\oo} x)\sgn(x)\right)
\]
and so (here $\ONE_{[1,2]}$ denotes the indicator function of the interval $[1,2]$)
\begin{equation}\label{:zeta}
\left|(\ln \zeta_A)''\right|
\leq \frac{C\oo^2}A \, \ONE_{[1,2]} (\sqrt{\oo} |x|)
\leq \frac{C}A \oo^2\rho^4(x).
\end{equation}
Thus,
\[
- \int \left(2 \Phi_A \px u_k+ \Phi_A' u_k\right) \px^2u_k 
\geq 2 \int ( \px w_k)^2 - \frac {C}{A} \oo^2 \|\rho^2 w_k \|^2.
\]
For the next term in the expression of~$\dot \cI$, we note that
\begin{multline*}
 \px \left( F(\po + u) - F(\po) -f(\po) u \right)
 =\Re \left\{( \px \bar u) (f(\po+u) - f(\po)) \right\}
\\ + \Re \left\{ \po' (f(\po+u) - f(\po)- f'(\po) u)\right\}
\end{multline*}
and so by integration by parts, we decompose
\begin{align*}
& - \Re \left\{ \int \left(2 \Phi_A \px \bar u + \Phi_A' \bar u \right) \left(f(\po+u)-f(\po) \right)\right\} \\
&\quad = 2 \int \Phi_A' \Re\left\{ F(|\po + u|) - F(\po) -f(\po) u \right\}\\
&\qquad + 2 \int \Phi_A \Re \left\{ \po' (f(\po+u) - f(\po)- f'(\po) u)\right\} \\
&\qquad - \int \Phi_A' \Re \left\{ \bar u \left(f(\po+u)-f(\po) \right)\right\} = I_1+I_2+I_3.
\end{align*}
To estimate $I_1$, $I_2$ and $I_3$, we use the following observations
\[
0<\Phi'_A(x)= \zeta_A^2\leq 1,\quad |\Phi_A(x)|\leq |x| \quad \hbox{on $\R$},
\]
and so (the estimate below is not optimal in terms of power of $\rho$)
\[
\left|\Phi_A(x) \po\right|
\leq \sqrt{\omega} |x| \sech\left(\sqrt{\omega} x\right)
\leq C \rho^4(x).
\]
Moreover, by Lemma~\ref{LE:1},
\[
\|u\|_{L^\infty}\leq C \|u\|_{H^1(\R)}\leq C \varepsilon\leq C \oo.
\]
Thus,
\begin{align*}
|I_1| &\leq C \int \zeta_A^2 \left( \po^2 |u|^2 + |u|^4 \right)
\leq C \oo \int \rho^4 |u|^2 + C \int \zeta_A^2 |u|^4,\\
|I_2| &\leq C \sqrt{\oo} \int |\Phi_A| \po \left( \po |u|^2 + |u|^3 \right)
\leq C \oo \int \rho^4 |u|^2 ,\\
|I_3| &\leq C \int \zeta_A^2 \left( \po^2 |u|^2 + |u|^4 \right)
\leq C \oo \int \rho^4 |u|^2 + C \int \zeta_A^2|u|^4.
\end{align*}
Using the following inequality from~\cite[Claim 1]{KMM}
\[
\int \zeta_A^2 |u|^4 
\leq \frac{C A^2}{\oo^3}\|u\|_{L^\infty}^2 \int | \px w|^2
\leq \frac{C A^2\varepsilon^2}{\oo^3} \int | \px w|^2
\]
we obtain
\[
|I_1|+|I_2|+|I_3| \leq C \oo \int \rho^4 |u|^2 + \frac{C A^2\varepsilon^2}{\oo^3}\int | \px w|^2.
\]
Next, by integration by parts and then using~\eqref{eq:modulation}, for $k=1,2$,
\begin{multline*}
\biggl| \int \left(2 \Phi_A \px u_k+ \Phi_A' u_k \right) \theta_k \biggr| 
 = \biggl| \int u_k\left( 2 \Phi_A \px \theta_k + \Phi_A' \theta_k \right) \biggr| \\
 \leq \|u\|_{L^\infty} \int \left(|x| |\px \theta_k|+|\theta_k|\right) 
\leq C \varepsilon \|\rho^2 u\|^2\leq C \oo \|\rho^2 u\|^2.
\end{multline*}
For the last term in the expression of $\dot \cI$, we integrate by parts
\begin{equation*}
-\int (2 \Phi_A \px u_1+ \Phi_A' u_1 ) m_1
= \dot\beta \int \Phi_A u_1^2
+(\dot \sigma-2\beta) \int (2 \Phi_A \px u_1+ \Phi_A' u_1 ) \px u_2.
\end{equation*}
Combining this identity with its analogue for
$\int (2 \Phi_A \px u_2+ \Phi_A' u_2 ) m_2$, we obtain
\begin{multline*}
-\int (2 \Phi_A \px u_1+ \Phi_A' u_1 ) m_1
-\int (2 \Phi_A \px u_2+ \Phi_A' u_2 ) m_2\\
 = \dot\beta \int \Phi_A |u|^2
+(\dot \sigma-2\beta) \int \Phi_A'(u_2 \px u_1-u_1 \px u_2).
\end{multline*}
Thus, using $\|\Phi_A\|_{L^\infty}+\|x\Phi_A'\|_{L^\infty}\leq C A \oo^{-\frac 32}$,
$|\Phi_A'|\leq1$, $\|u\|_{H^1(\R)}\leq C \varepsilon$
and~\eqref{eq:modulation},
\begin{equation*}
\biggl| \sum_{k=1,2} \int (2 \Phi_A \px u_k+ \Phi_A' u_k ) m_k\biggr|
\leq \frac {CA}{\oo^{\frac 12}}\varepsilon^2 \|\rho^2u\|^2.
\end{equation*}

Gathering these estimates, we have proved
\[
\dot{\cI} \geq 2\bigg( 1- \frac {C A^2 \varepsilon^2}{\oo^3}\bigg)\int |\px w|^2
-C \bigg(\oo+\frac {A\varepsilon^2}{\oo^{\frac 12}}\bigg) \|\rho^2u\|^2.
\]
Taking $\varepsilon$ small such that
\begin{equation}\label{333}
\frac {C A^2 \varepsilon^2}{\oo^3}\leq \frac 12,
\end{equation}
the estimate~\eqref{eq:dtI} is proved.

Now, for any $T\geq0$, using the above estimates for $\Phi_A$ and \eqref{eq:small}, we estimate
\[
|\cI(T)| \leq \left(\|\Phi_A\|_{L^\infty}+\|\Phi_A'\|_{L^\infty}\right)\|u\|_{H^1(\R)}^2
\leq \frac{C A}{\oo^\frac 32} \varepsilon^2\leq C \varepsilon.
\]
Therefore, integrating on $[0,T]$, we obtain
\[
\int_0^T \int | \px w|^2 \leq C\varepsilon + C\oo \int_0^T \|\rho^2u\|^2 .
\]
Now, we use the elementary inequality (see \emph{e.g.}~\cite[Lemma 4]{KMM})
\[
\int \eta_A|w|^2 \leq \frac {CA^2}{\oo^3} \int | \px w|^2
+ \frac{C A}\oo \int \rho^4 |w|^2 ,
\]
which implies
\[
\frac {\oo^3}{A^2} \int_0^T \int \eta_A^2 |u|^2 
\leq C\varepsilon + \frac{C\oo^2}{A} \int_0^T \|\rho^2u\|^2 .
\]
Last, recalling $w=\zeta_A u$, by integration by parts,
\[
\int \zeta_A^2|\px w|^2 
= \int \zeta_A^4 |\px u|^2 - \int \zeta_A^3 \zeta_A'' |u|^2
-2 \int \zeta_A^2 (\zeta_A')^2|u|^2 
\]
and so using $\frac 1C\eta_A \leq \zeta_A^2\leq C \eta_A$
and $|\zeta_A''|+|\zeta_A'|^2\leq C \oo^3 A^{-2}\zeta_A$,
\[
\int \eta_A^2| \px u|^2 
\leq C \int | \px w|^2 + \frac {C \oo^3}{A^2} \int \eta_A^2|u|^2 ,
\]
which is sufficient to complete the proof.
\end{proof}

\subsection{Transformed problem}\label{S:transformed}
For~$\alpha>0$ small to be fixed,
we introduce the function $v=v_1 + \ii v_2$ defined by
\[
v_1 = \Xe^2\MM S^2 u_2
\AND
v_2 = - \Xe^2 S^2 \LP u_1.
\]

By direct computations, using
\[
S^2 = \px^2 - 2\frac{\po'}{\po} \px + \omega- \frac 13 \po^4
\]
and the identities~\eqref{eq:id}, we have
\begin{align*}
\MM S^2 
& = - \px^4 + 2 \px^2 \cdot\frac{\po'}{\po} \cdot \px +\frac 43 \px \cdot \po^4 \cdot \px+ \left(-2 \omega \frac {\po'}{\po} - \frac {14}3 \po^3 \po'\right)\cdot \px\\
&\quad + \omega^2 + 6 \omega \po^4 - \frac {10}3 \po^6 + \frac 73 \po^8
\end{align*}
and
\begin{align*}
S^2 \LP 
& = - \px^4 + 2 \px^2 \cdot\frac{\po'}{\po} \cdot \px +\px \cdot \left(-\po^2+\frac 83 \po^4 \right)\cdot \px\\
&\quad + \left(-2 \omega \frac {\po'}{\po}-2 \po\po' +14\po^3 \po'\right)\cdot \px\\
&\quad + \omega^2 -3\omega \po^2+3\po^4+ \frac{134}3 \omega \po^4 - 38 \po^6 + 25 \po^8.
\end{align*}
For future use, we introduce \begin{align*}
\QM
& = 2 \px^2 \cdot\left(\frac{\Lo'\po-\po'\Lo}{\po^2}\right) \cdot \px +\frac {16}3\px \cdot (\Lo\po^3) \cdot \px\\
&\quad + \left[-2 \omega\frac {\po'}{\po}-2 \omega \left(\frac{\Lo'\po-\po'\Lo}{ \po^2}\right)
- \frac {14}{3} \left(3\po^2\Lo\po'+ \po^3 \Lo' \right) \right]\cdot \px \\
&\quad + 2 \omega^2 + 6 \omega\po^4 + 24\omega \Lo \po^3- \frac {60}{3} \Lo \po^5 + \frac {56}{3} \Lo\po^7
\end{align*}
and
\begin{align*}
\QP
& = 2 \px^2 \cdot\left(\frac{\Lo'\po-\po'\Lo}{\po^2}\right)\cdot \px 
+\px \cdot \left(-2\Lo \po+\frac {32}3 \Lo \po^3 \right)\cdot \px\\
&\quad + \left[-2 \omega\frac {\po'}{\po}-2 \omega\left(\frac{\Lo'\po-\po'\Lo}{ \po^2}\right)\right]\cdot \px\\
&\quad + \left( -2 \Lo\po'-2 \po \Lo' +42 \Lo \po^2 \po' + 14\po^3\Lo'\right)\cdot \px\\
&\quad + 2 \omega^2 -3\omega \po^2-6 \omega\Lo \po +12 \Lo\po^3+ \frac{134}{3}\omega \po^4
+ \frac{536}3 \omega\Lo \po^3\\
&\quad - 228   \Lo\po^5 + 200  \Lo\po^7.
\end{align*}
Note that the operators $\QM$ and $\QP$ are obtained from $\MM S^2$ and $S^2\LP$ by differentiation with respect to $\omega$. Their exact expressions are not so important, only their specific structures (similar to the ones of 
$\MM S^2$ and $S^2\LP$) are used in the proof of Lemma~\ref{LE:8}.

We recall some technical estimates from~\cite{KMMV}.
\begin{lemma}\label{LE:6}
There exists~$C>0$ such that 
for any~$\alpha> 0$ small and~$h\in L^2(\R)$
\begin{align*}
\|\Xe h\| \leq \|h\|,&\quad
\| \px \Xe^\frac12 h\| \leq \alpha^{-\frac 12}\|h\|,\\
\|\rho \Xe h\| \leq C \|\Xe [\rho h]\|,&\quad 
\|\rho^{-1} \Xe [\rho h]\| \leq C \|\Xe h\|,\\
\|\eta_A \Xe h\| \leq C \|\Xe [\eta_A h]\|,&\quad 
\|\eta_A^{-1} \Xe [\eta_A h]\| \leq C \|\Xe h\|,\\
\|\rho^{-1} \Xe \px^2[\rho h]\| \leq C \alpha^{-1}\| h\|,&\quad
\|\rho^{-1} \Xe \px[\rho h]\| \leq C \alpha^{-\frac 12}\|h\|,\\
\|\eta_A \Xe \px^2h\| \leq C \alpha^{-1}\|\eta_A h\|,&\quad
\|\eta_A \Xe \px h\| \leq C \alpha^{-\frac 12}\|\eta_A h\|.
\end{align*}
\end{lemma}
\begin{proof}
These estimates follow directly from~\cite[Lemma 4.7]{KMMV}
(see also~\cite[Lemma 5]{KMM}) except the 
last two lines. We prove the estimates $\|\eta_A \Xe \px^2h\|$ and $\|\eta_A \Xe \px h\|$.
First,
\[
\|\eta_A \Xe \px^2 h\|
\leq \|\Xe [\eta_A \px^2 h]\|.
\]
Using~$\eta_A \px^2 h= \px^2(\eta_A h) - 2 \px(\eta_A' h)+\eta_A'' h$
and $|\eta_A'|+|\eta_A''|\leq C \eta$, we have
\begin{equation*}
\|\Xe [\eta_A \px^2 h]\|
\leq \alpha^{-1} \|\eta_A h\|
+ 2\alpha^{-\frac 12 } \|\eta_A' h\|
 + \|\eta_A''h\|
\leq C \alpha^{-1} \|\eta_Ah\|.
\end{equation*}
Similarly,
\begin{equation*}
\| \Xe [\eta_A \px h]\|
\leq \| \Xe [ \px (\eta_Ah)]\|
+ \| \Xe [\eta_A'h]\|
\leq C \alpha^{-\frac 12}\| \eta_Ah\|.
\end{equation*}
The estimates on $\|\rho^{-1} \Xe \px^2[\rho h]\|$, $\|\rho^{-1} \Xe \px[\rho h]\|$
are proved similarly.
\end{proof}

\begin{lemma}\label{LE:7}
There exists~$C>0$ such that for any~$\alpha>0$ small and $g\in H^1(\R)$
\begin{align*}
\| \eta_A \Xe^2 \MM S^2g \|+\| \eta_A \Xe^2S^2 \LP g \| & \leq C \big(\alpha^{-\frac 32} \|\eta_A \px g\|+\oo^2\|\eta_A g\|\big),\\
\| \eta_A \px \Xe^2 \MM S^2g\|+\| \eta_A \px \Xe^2S^2 \LP g\| & \leq C \big(\alpha^{-2} \|\eta_A \px g\|+ \oo^\frac 52\|\rho^2 g\|\big).
\end{align*}
\end{lemma}
\begin{proof}
We prove the two estimates for $\Xe^2 \MM S^2g$. The proof for $\Xe^2S^2 \LP g$ is identical.
By Lemma~\ref{LE:6},
\begin{equation*}
\|\eta_A \Xe^2 \px^4 g\|
\leq C \alpha^{-\frac 32}\| \eta_A \px g\|,\quad 
\|\eta_A \Xe^2 \px^5 g\| 
\leq C \alpha^{-2}\| \eta_A \px g\|.
\end{equation*}
By Lemma~\ref{LE:6} and~$| \frac{\po'}{\po}|\leq C\sqrt{\oo} \leq C$ (see \eqref{eq:id}), we have
\begin{align*}
\Big\| \eta_A \Xe^2 \px^2 \cdot \frac{\po'}{\po}\cdot \px g \Big\|
&\leq C \alpha^{-1} \|\eta_A \px g\|,\\
\Big\| \eta_A \Xe^2 \px^3 \cdot \frac{\po'}{\po}\cdot \px g \Big\|
&\leq C \alpha^{-\frac 32} \|\eta_A \px g\|.
\end{align*}
Similarly,
\begin{align*}
\|\eta_A \Xe^2 \px \cdot \po^4 \cdot \px g\|
&\leq C \alpha^{-\frac 12}\| \eta_A \px g\|,\\
\|\eta_A \Xe^2 \px^2 \cdot\po^4 \cdot \px g\|
&\leq C \alpha^{-1}\| \eta_A \px g\|,
\end{align*}
and
\begin{align*}
\Big\|\eta_A \Xe^2 \Big(2 \omega \frac {\po'}{\po} +\frac {14}3 \po^3 \po'\Big) \cdot \px g\big\|
&\leq C \| \eta_A \px g\|,\\
\Big\|\eta_A \Xe^2 \px\Big(2 \omega \frac {\po'}{\po} + \frac {14}3 \po^3 \po'\Big) \cdot \px g\Big\|
&\leq C \alpha^{-\frac 12}\| \eta_A \px g\|
\end{align*}
Moreover,
\[
\Big\|\eta_A \Xe^2 \Big(\omega^2 + 6 \omega \po^4 - \frac {10}3 \po^6 + \frac 73 \po^8\Big)g\Big\|
\leq C \oo^2\|\eta_A g\|.
\]
Last, we observe that
\begin{align*}
 \px \left[ \left(\omega^2 + 6 \omega \po^4 - \frac {10}3 \po^6 + \frac 73 \po^8\right) g\right]
&= \left(\omega^2 + 6 \omega \po^4 - \frac {10}3 \po^6 + \frac 73 \po^8\right) \px g
\\&\quad + \left(24 \omega \po^3 - 20 \po^5 + \frac {56}3 \po^7\right)\po'g.
\end{align*}
As before,
\[
\Big\|\eta_A \Xe^2 \Big(\omega^2 + 6 \omega \po^4 - \frac {10}3 \po^6 + \frac 73 \po^8\Big) \px g\Big\|
\leq C \| \eta_A \px g\|
\]
and since~$|\po'|\leq C \rho^2$,
\[
\Big\|\eta_A\Xe^2 \Big(24 \omega \po^3 - 20 \po^5 + \frac {56}3 \po^7\Big) \po' g\Big\|
\leq C \oo^{\frac 72}\| \rho^2 g\|.
\]
This completes the proof of Lemma~\ref{LE:7}.
\end{proof}

Applying Lemma~\ref{LE:7} to $u_2$ and $u_1$, we obtain the following result.
\begin{lemma}\label{LE:9}
There exists~$C>0$ such that for any~$\alpha>0$ small,
\begin{align*}
\| \eta_A v \| & \leq C \big(\alpha^{-\frac 32} \|\eta_A \px u\|+\oo^2\|\eta_A u\|\big),\\
\| \eta_A \px v\| & \leq C \big(\alpha^{-2} \|\eta_A \px u\|+ \oo^\frac52\|\rho^2 u\|\big).
\end{align*}
\end{lemma}

\begin{lemma}\label{LE:8}
There exists~$C>0$ such that for any~$\alpha>0$ small and $g\in H^1(\R)$
\begin{align*}
\| \eta_A \Xe^2 \QM g \|+\| \eta_A \Xe^2\QP g \| & \leq C \big(\alpha^{-1}\oo^{\frac 12} \|\eta_A \px g\|+\oo^2\|\eta_A g\|\big),\\
\| \eta_A \px \Xe^2 \QM g\|+\| \eta_A \px \Xe^2\QP g\| & \leq C \big(\alpha^{-\frac 32}\oo^{\frac 12} \|\eta_A \px g\|+ \oo^\frac 52\|\rho^2 g\|\big).
\end{align*}
\end{lemma}
\begin{proof}
The proof is similar to the one of Lemma~\ref{LE:7}.
\end{proof}

\subsection{Second virial estimate}
From the equation~\eqref{eq:u} of~$u$ and the identity of Lemma~\ref{LE:2},~$v$ satisfies
\begin{equation}\label{eq:v}
\begin{cases}
\pt v_1 = \MM v_2 + \Ye v_2 +\Xe^2 n_2 - \Xe^2 r_2 \\
\pt v_2 = -\MP v_1 + \frac 13 \Ye v_1 -\Xe^2 n_1 + \Xe^2 r_1 
\end{cases}
\end{equation}
where we have used~$S^2 \theta_1=S^2 \LP \theta_2=0$ 
(see \S \ref{S:2.2} and \eqref{eq:S2})
and the notation
\begin{gather*}
n_1 = S^2 \LP m_2+ \frac{\dot \omega}{\omega}\QP u_1,
\quad n_2 = -\MM S^2 m_1+ \frac{\dot \omega}{\omega}\QM u_2,\\
r_1 = S^2 \LP q_2, \quad r_2 = -\MM S^2 q_1,
\end{gather*}
and
\begin{equation*}
Y_\alpha=\Xe^2 \cdot \po^4 \cdot \Xe^{-2}-\po^4.
\end{equation*}

\begin{remark}\label{RK:3}
The identity used for the transformed problem is different from the one for the wave-type equations in \cite{KMM,KMMV} and for the Schr\"odinger equation with a real potential in \cite{CM2}. However, the underlying idea is similar: the system for $(v_1,v_2)$ has the same structure as the original system in $(u_1,u_2)$, but with more favorable operators $\MP$, $\MM$.
\end{remark}

The next proposition provides the key estimate on the function $v$.
Thanks to the special structure of the operators $\MP$ and $\MM$, a virial estimate on $v$
controls a weighted $L^2$ norm of $v$ by higher order terms in $u$.

\begin{proposition}\label{PR:3}
Assume that $\oo\in (0,\frac 18]$.
There exists~$C>0$ such that, for~$B>0$ large enough, $\alpha>0$ and
$\varepsilon>0$ small enough, for any~$T>0$,
\[
\oo^2 \int_0^T \|\rho v \|^2\ud t
\leq C\varepsilon + \frac CA \int_0^T\left( \|\eta_A\px u\|^2 + \frac{\oo^3}{A^2}\|\eta_A u\|^2
+ \oo\|\rho^2 u\|^2\right)\uds t.
\]
\end{proposition}
\begin{proof}
First, we introduce
\[
\cJ = \int v_1 \left( 2 \Psi_{A,B} \px v_2 + \Psi_{A,B}' v_2\right)\quad \mbox{and}\quad
z=\chi_A \zeta_B v.
\]
By the equation~\eqref{eq:v} of $v$ and direct computations (see~\cite[Proof of Proposition~2, \S 4.3]{KMM}),
we compute
\[
\dot \cJ = 
\int \left( 2 ( \px z_1)^2 - P_B z_1^2\right)
+ \int \left( 2 ( \px z_2)^2 + 3 P_B z_2^2\right) + \sum_{j=1}^5 J_j,
\]
where
\[
P_B = - \frac 13 \frac{\Phi_B}{\zeta_B^2} (\poo^4)' ,
\]
and
\begin{align*}
J_1 &= \sum_{k=1}^2 \int (\ln \zeta_B)'' z_k^2,\\
J_{2}
& = - \sum_{k=1}^2 \int \left( \frac 12 (\chi_A^2)'(\zeta_B^2)' + \left(3(\chi_A')^2+\chi_A''\chi_A\right)\zeta_B^2 + \frac 12 (\chi_A^2)''' \Phi_B \right) v_k^2 \\ 
& \quad + 2 \sum_{k=1}^2 
\int (\chi_A^2)' \Phi_B( \px v_k)^2,\\
J_3 & = -\frac 13 \int \left(2 \Psi_{A,B} \px v_1+ \Psi_{A,B}' v_1\right) \Ye v_1
+ \int \left(2 \Psi_{A,B} \px v_2+ \Psi_{A,B}' v_2\right) \Ye v_2,\\
J_4 & = \sum_{k=1}^2 \int \left(2 \Psi_{A,B} \px v_k+ \Psi_{A,B}' v_k\right)(\Xe^2 n_k- \Xe^2 r_k)\\
J_5 & = - \frac 13 \int \frac{\Phi_B}{\zeta_B^2} (\poo^4-\po^4)'( z_1^2 - 3z_2^2).
\end{align*}

Second, we set
\[
\cK = - \int z_1 z_2 R_B,
\]
where the function~$R_B$ is the unique (time-independent) bounded solution of
\[
- \frac 12 R_B'' + \oo R_B = \frac 32 P_B\quad \mbox{on~$\R$.}
\]
Using the equation~\eqref{eq:v} of $v$, integrating by parts and using the equation of~$R_B$, we obtain
\[
\dot K = \frac32\int (z_1^2-z_2^2) P_B + \sum_{j=1}^5 K_j,
\]
where
\begin{align*}
K_1 &= \sum_{k=1}^2 (-1)^k \int v_k^2 \left( (\chi_A \zeta_B)'\chi_A \zeta_B R_B' + \left((\chi_A\zeta_B)'\right)^2 R_B \right)\\
K_2 &= \int \left(( \px z_1)^2 - ( \px z_2)^2\right) R_B
- \int \left( \frac 13 z_1^2 + z_2^2 \right) \po^4 R_B\\
K_3 &= - \int \left(\frac 13 (\Ye v_1) v_1 + (\Ye v_2) v_2 \right) \chi_A^2 \zeta_B^2 R_B,\\
K_4 &= \sum_{k=1}^2 (-1)^{k-1}\int (\Xe^2 n_k-\Xe^2 r_k) v_k \chi_A^2 \zeta_B^2 R_B,\\
K_5 &= (\omega-\oo) \int (z_1^2 - z_2^2) R_B.
\end{align*}
Therefore,
\begin{align}
\dot \cJ +\dot \cK & = 
\int \left(2 ( \px z_1)^2 + \frac 12 P_B z_1^2\right)
+ \int \left(2 ( \px z_2)^2 + \frac 32 P_B z_2^2\right)\label{111}\\
&\quad +\sum_{j=1}^5 J_j+ \sum_{j=1}^5 K_j.\nonumber
\end{align}
Since $P_B$ is nonnegative and not identically zero on $\R$, the functional $\cJ+\cK$ will allow us to estimate
the function $z$ provided that the error terms~$J_j$ and~$K_j$ are controlled. 
Comparing the expressions of $\dot\cJ$ and $\dot\cJ+\dot\cK$, we observe that the functional $\cK$ was used to exchange 
$\frac 32 \int P_B z_2^2$ by $\frac 32 \int P_B z_1^2$ at the cost of additional error terms.

To estimate the error terms, we need several technical estimates.

\begin{lemma}\label{LE:3}
For any~$\oo\in (0,\frac 18]$,
\[
\|P_B\|_{L^\infty} \leq \frac 15\oo ,\quad 
\|R_B\|_{L^\infty} \leq \frac 7{18}.
\]
Moreover,
\[
0\leq P_B(x) \leq C\oo \poo^2,\quad 0\leq R_B(x) \leq C\poo^2
\quad 
|R_B'(x)|\leq C \sqrt{\oo} \poo^2 \quad \hbox{on~$\R$}.
\]
\end{lemma}
\begin{proof}
Using~$|\Phi_B|\leq |x|$ and~$\zeta_B(x)\geq \ee^{-\frac{\oo\sqrt{\oo}}B |x|} \geq \ee^{-\frac{\sqrt{\oo}}{8B} |x|}$, we have
\begin{equation}\label{eq:PB}
0\leq P_B(x) \leq \frac 13 |x| \ee^{\frac{\sqrt{\oo}}{4B} |x|} |(\poo^4)'|.
\end{equation}
Note that
\[
\poo^4 = \frac {16\oo^2}{\left( 1+a_{\oo} \cosh(2 \sqrt{\oo} x) \right)^2},\quad
(\poo^4)' = - a_{\oo} \oo^{-\frac 12} \poo^6 \sinh(2\sqrt{\oo} x).
\]
Moreover,~$\oo\leq \frac 18$ implies~$a_{\oo}\geq \frac 1{\sqrt3}$.
By the inequality \eqref{eq:A1} proved in Appendix~A, we obtain $0\leq P_B(x) \leq \frac {32}9\oo^2 a_{\oo}$.
For any~$\oo \in (0,\frac 18]$, we have~$\frac {32}9 \oo a_{\oo}\leq\frac 4{9\sqrt{3}}\leq \frac 7{27}$ and thus
$0\leq P_B(x) \leq \frac 7{27}\oo$. 
From \eqref{eq:PB} and using first $B\geq 1$, $\frac1{\sqrt 3}\leq a_\omega\leq 1$ and then
$|\sinh y|\leq \cosh y$ and $|y|e^{\frac{|y|}{8}}\leq C \cosh(y)$, we have
\[
P_B(y)\leq \frac{16}3\oo\po^2\frac{\sqrt{\oo} |x| e^{\frac{\sqrt{\oo}}{4}|x|} |\sinh(2\sqrt{\oo} x)|}
{(1+\frac 1{\sqrt 3}\cosh(2\sqrt{\oo} x))^2}
\leq C \oo \poo^2.
\]
Next, we recall the explicit expression of $R_B$
\begin{equation*}
R_B =\frac 32 \frac 1{\sqrt{2\oo}} \left( \int_{-\infty}^x \ee^{\sqrt{2\oo} (y-x)} P_B(y) \ud y
+\int_x^\infty \ee^{\sqrt{2\oo} (x-y)} P_B(y) \ud y\right)
\end{equation*}
and so~$\|R_B\|_{L^\infty} \leq \frac 32\oo^{-1} \|P_B\|_{L^\infty}
\leq \frac 7{18}$.
Inserting $0\leq P_B(x) \leq C\oo^2 \ee^{-2\sqrt{\oo} |x|}$ into the above expression of $R_B$, we also find
$0\leq R_B(x)\leq C \oo \ee^{-2\sqrt{\oo} |x|}$ and
$|R_B'(x)|\leq C \oo^\frac 32 \ee^{-2\sqrt{\oo} |x|}$.
\end{proof}

\begin{lemma}\label{LE:4}
For any~$\oo\in (0,\frac 18]$,
for any function~$h\in H^1(\R)$,
\[
\int \poo^4 h^2 \leq \frac {19}{6} \int P_B h^2 + 3\int (h')^2.
\]
\end{lemma}
\begin{proof}
By the definition of~$P_B$ and integration by parts, we have
\[
3 \int P_B h^2 = - \int \frac{\Phi_B}{\zeta_B^2} (\poo^4)' h^2
=  \int \left(\frac{\Phi_B}{\zeta_B^2}\right)' \poo^4 h^2
+2 \int \frac{\Phi_B}{\zeta_B^2} \poo^4 h h'.
\]
Using
\[
\int \left(\frac{\Phi_B}{\zeta_B^2}\right)' \poo^4 h^2
=\int \poo^4 h^2 - 2 \int \frac{\Phi_B\zeta_B'}{\zeta_B^3} \poo^4 h^2
\geq \int \poo^4 h^2
\]
and
\[
2 \left|\int \frac{\Phi_B}{\zeta_B^2} \poo^4 h h'\right|
\leq \frac 13 \int \frac{\Phi_B^2}{\zeta_B^4} \poo^8 h^2 + 3\int (h')^2,
\]
we obtain
\[
\int \poo^4 h^2 \leq \left( 3+ \frac 13 \left\|\frac{\Phi_B^2\poo^8 }{\zeta_B^4P_B} \right\|_{L^\infty} \right)\int P_B h^2 + 3 \int (h')^2.
\]
We claim the following estimate 
\[
\left\|\frac{\Phi_B^2\poo^8 }{\zeta_B^4P_B}\right\|_{L^\infty}
\leq 4 \oo \leq \frac 12
\]
which is sufficient to complete the proof. To prove the above estimate, we recall that 
$a_{\oo}\geq \frac 1{\sqrt3}$, $|\Phi_B|\leq |x|$, $\zeta_B\geq e^{-\frac{\sqrt{\oo} x}{8B}}$, so that
\begin{align*}
\left| \frac{\Phi_B^2\poo^8 }{\zeta_B^4P_B} \right|
= \frac{3 |\Phi_B| \poo^8}{\zeta_B^2 |(\poo^4)'|}
&=\frac{3 \sqrt{\oo}|\Phi_B| \poo^2}{\zeta_B^2 a_{\oo} |\sqrt{\oo}x\sinh(2\sqrt{\oo}x)|}\\
&\leq {6\oo} \biggl\| \frac{y \ee^{\frac1{4B} |y|}}{\sinh(y)(1+ \frac 1{\sqrt3} \cosh(y))}\biggr\|_{L^\infty}
\leq 4 \oo
\end{align*}
using in the last step the inequality \eqref{eq:A3} proved in Appendix A.
\end{proof}

\begin{lemma}\label{LE:5}
There exists~$c>0$ such that, for any~$\oo\in (0,\frac 18]$, for any~$x\in \R$,
\[
P_B(x) \geq c \oo^2 \ONE_{[1,2]}(\sqrt{\oo} |x|).
\]
Moreover, there exists $C>0$ such that for any $h\in H^1(\R)$,
\begin{equation}\label{eq:std}
\oo^2 \int \rho h^2 \leq C\oo \int (h')^2 + C\int P_B h^2.
\end{equation}
\end{lemma}
\begin{proof}
The function $\zeta:[0,\infty)\to [0,1]$ being non increasing, 
for $x\geq 0$, $\Phi_B(x) = \int_0^x \zeta_B^2 \geq |x| \zeta_B^2(x)$.
By parity $|\Phi_B(x)|\geq |x| \eta_B^2(x)$ on $\R$. Thus,
\begin{equation*}
P_B = - \frac 13 \frac{\Phi_B}{\zeta_B^2} (\poo^4)'
\geq c \oo^2 \frac{\sqrt{\oo} |x| |\sinh(2\sqrt{\oo} x)|}{(1+\cosh(2\sqrt{\oo} x))^3}\geq c \oo^2 \ONE_{[1,2]}(\sqrt{\oo} |x|),
\end{equation*}
where~$c$ denotes positive constants.
The inequality~\eqref{eq:std} then follows from standard arguments, see \emph{e.g.}~\cite[Lemma 4]{KMM}.
\end{proof}

Last, we prove an estimate on $v$ in terms of $z$, plus an error term in $u$.
\begin{lemma}\label{LE:10} There exists~$C>0$ such that
\begin{align*}
\int \rho^2 (|\px v|^2 + \omega_0^2 |v|^2)
&\leq C \int \left(|\px z|^2 + P_B |z|^2\right)\\
&\quad +\frac CA \left( \alpha^{-4}\|\eta_A\px u\|^2 + \frac{\oo^3}{A^2}\|\eta_A u\|^2\right).
\end{align*}
\end{lemma}
\begin{proof}
We claim
\[
\int_{\oo^{\frac 32}|x|\leq A} \rho^2 (|\px v|^2 + \omega_0^2 |v|^2)
\leq C \int \left(|\px z|^2 + P_B |z|^2\right).
\]
By definition, for~$\oo^\frac 32|x|\leq A$, one has~$z = \zeta_B v$ and so (for $B$ large)
\[
\int_{\oo^\frac 32|x|\leq A} \rho^2 |v|^2 \leq C\int_{\oo^\frac 32|x|\leq A}\rho \, \zeta_B^2 |v|^2 
\leq C\int_{\oo^\frac 32|x|\leq A} \rho |z|^2 .
\]
For $\oo^\frac 32|x|\leq A$, using $\px z = \zeta_B' v + \zeta_B \px v$
and $|\zeta_B'|\leq C \oo^{\frac 32} B^{-1} \zeta_B$, we also have
\begin{align*}
\rho^2 |\px v|^2 
\leq C \rho \zeta_B^2 |\px v|^2
&\leq C \rho |\px z|^2 + C\oo^3 B^{-2} \rho \zeta_B^2 |v|^2 \\
&\leq C \rho |\px z|^2 + C\oo^3 B^{-2} \rho |z|^2
\end{align*}
and so
\[
\int_{\oo^\frac 32|x|\leq A} \rho^2 |\px v|^2 \leq C\int |\px z|^2+ C\frac{\oo^3}{B^2}\int \rho |z|^2.
\]
We complete the proof of the claim using \eqref{eq:std}.

Now, using Lemma~\ref{LE:9},
\begin{align*}
\int_{\oo^{\frac 32}|x|\geq A} \rho^2 (|\px v|^2 +|v|^2)
&\leq C\ee^{-\frac A{10 \oo}} \left(\|\eta_A \px v\|^2 +\|\eta_A v\|^2\right)\\
&\leq C\frac{\oo^3}{A^3} \left(\alpha^{-4}\|\eta_A \px u\|^2 +\|\eta_A u\|^2\right),
\end{align*}
which implies the desired estimate.
\end{proof}

\emph{Estimate of~$J_1$.}
By~\eqref{:zeta} and then Lemma~\ref{LE:5}, we have
\[
\left|(\ln \zeta_B)''\right|
\leq \frac{C \oo^2}B \ONE_{[1,2]}(\sqrt{\oo} |x|)
\leq \frac CB P_B .
\]
Thus, for~$B$ large enough (independent of~$\alpha$, $\oo$,~$A$ and~$\varepsilon$), we have
\[
|J_1| \leq \frac 1{100} \int P_B |z|^2.
\]

\emph{Estimate of~$K_1$.}
Using $|\chi_A'|\leq C \oo^{\frac 32}A^{-1}\leq C \oo^{\frac 32}B^{-1}$, $|\zeta_B'|\leq C \oo^{\frac 32}B^{-1}\zeta_B$,
and Lemma~\ref{LE:3}, we estimate
\[
\left|(\chi_A \zeta_B)'\chi_A \zeta_B R_B' + \left((\chi_A\zeta_B)'\right)^2 R_B \right|
\leq \frac{\oo^3}B \rho^2.
\]
Thus,
\begin{equation*}
|K_1| 
\leq \frac{C\oo^3}B \int \rho^2 |v|^2,
\end{equation*}
and using Lemma~\ref{LE:10}, taking~$B$ large enough (independent of~$\alpha$, $\oo$, $A$ and~$\varepsilon$),
\[
|K_1|\leq \frac 1{100} \int \left( |\px z|^2 + P_B |z|^2\right)
+\frac CA \left( \alpha^{-4}\|\eta_A\px u\|^2 + \frac{\oo^3}{A^2}\|\eta_A u\|^2\right).
\]
From now on,~$B$ is fixed so that the above estimates on $J_1$ and $K_1$ hold.

\emph{Estimate of~$J_2$.}
First we recall some bounds on the functions involved in the definition of~$J_2$.
We have
\[
|\chi_A'|\leq \frac {C\oo^\frac 32}{A},\quad |\chi_A''|\leq \frac {C\oo^3}{A^2},
\quad |\chi_A'''|\leq \frac {C\oo^\frac 92}{A^3}
\]
and
\[
\chi_A'(x)=\chi_A''(x)=\chi_A'''(x)=0 \quad \hbox{if~$\oo^\frac 32|x|<A$ or if~$\oo^\frac 32|x|>2A$.}
\]
Moreover, 
\[
|\zeta_B(x)|\leq C\ee^{-\frac AB},\quad |\zeta_B'(x)|\leq \frac {C\oo^\frac 32}B\ee^{-\frac AB} \quad \hbox{for~$\oo^\frac 32|x|>A$.}
\]
Thus,
\[
|(\chi_A^2)'(\zeta_B^2)'|\leq \frac{C\oo^3}{AB}\ee^{-\frac AB} \eta_A^2,\quad
((\chi_A')^2+|\chi_A''\chi_A|)\zeta_B^2\leq \frac{C\oo^3}{A^2}\ee^{-\frac AB} \eta_A^2.
\]
Using also~$|\Phi_B|\leq CB\oo^{-\frac 32}$, we obtain
\[
|(\chi_A^2)'\Phi_B|\leq \frac{CB}{A}\eta_A^2,\quad
|(\chi_A^2)'''\Phi_B|\leq \frac{CB}{A^3}\oo^3\eta_A^2.
\]
Therefore (recall that~$B$ has been fixed)
\[
|J_2|\leq \frac CA \left( \|\eta_A \px v\|^2
+ \frac{\oo^3}{A^2}\|\eta_A v\|^2\right).
\]
Using Lemma~\ref{LE:9}, it follows that
\[
|J_2|\leq \frac CA \left(\alpha^{-4} \|\eta_A \px u\|^2+\frac{\oo^3}{A^2}\|\eta_A u\|^2
+ \oo\|\rho^2 u\|^2\right).
\]

\emph{Estimate of~$K_2$.}
Using Lemmas~\ref{LE:3},~\ref{LE:4} and~\eqref{eq:std}, we have
\begin{align*}
|K_2| & \leq \| R_B \|_{L^\infty} \left( \int |\px z|^2 
+ \frac 13 \int z_1^2 \poo^4 + \int z_2^2 \poo^4 +\int |z|^2 |\poo^4-\po^4|\right)\\
& \leq (1+C\varepsilon)\int \left( \frac 79 ( \px z_1)^2 + \frac {133}{324} P_B z_1^2
+ \frac {14}9 ( \px z_2)^2 + \frac {133}{108} P_B z_2^2\right)\\
& \leq \frac 9{10} \int \left(2 ( \px z_1)^2 + \frac 12 P_B z_1^2 
+2 ( \px z_2)^2 + \frac 32 P_B z_2^2\right)
\end{align*}
for $\varepsilon$ small enough.

\emph{Estimate of~$J_3$ and $K_3$.}
By the Cauchy-Schwarz inequality, and the bounds
\[
|\Psi_{A,B}|\leq C B \oo^{-\frac 32}\leq C \oo^{-\frac 32},\quad
|\Psi_{A,B}'|\leq C,
\]
we have for $k=1,2$
\[
\left| \int (2\Psi_{A,B} \px v_k + \Psi_{A,B}' v_k) \Ye v_k\right|
\leq C \Big( \oo^{-\frac 32} \|\rho \px v_k\|+\|\rho v_k\| \Big)\|\rho^{-1} \Ye v_k\|. 
\]
We rewrite~$\Ye$ as
\begin{align*}
\Ye 
&= 2 \alpha \Xe^2 \big[ 2 \px \cdot (\po^4)'- (\po^4)''\big]\\
& \quad + \alpha^2 \Xe^2\big[ - 4 \px^3\cdot (\po^4)' + 6 \px^2\cdot (\po^4)''
-4 \px\cdot(\po^4)'''- 2 (\po^4)^{(4)}\big].
\end{align*}
Using Lemma~\ref{LE:1} and Lemma~\ref{LE:6}, we obtain
\[
\|\rho^{-1} \Ye v_k\| \leq C\alpha^{\frac 12} \oo^{\frac 52}\|\rho v_k\|.
\]
Thus,
\[
|J_3|\leq C \alpha^\frac12 \oo \Big(\|\rho \px v\|+\oo^{\frac 32}\|\rho v\| \Big)\|\rho v\|.
\]
Therefore, using Lemma~\ref{LE:10},
\begin{equation*}
|J_3|
\leq C \alpha^\frac12 \int \left(| \px z|^2 + P_B |z|^2\right)
+ \frac C A \left(\alpha^{-4} \|\eta_A \px u\|^2
+\frac{\oo^3}{A^2} \|\eta_A u\|^2\right).
\end{equation*}
The estimate of~$K_3$ is similar and easier.

We fix $\alpha>0$ (independent of $\oo$, $A$ and $\varepsilon$) so that
\[
|J_3|+|K_3|
\leq \frac 1{100} \int \left(| \px z|^2 + P_B |z|^2\right)
+ \frac C A \left(\|\eta_A \px u\|^2
+\frac{\oo^3}{A^2} \|\eta_A u\|^2\right).
\]

\emph{Estimate of~$J_4$ and $K_4$.}
Using Lemma~\ref{LE:7} and Lemma~\ref{LE:8}, we have for $k=1,2$,
\[
 \|\eta_A \Xe^2n_k\|\leq 
C \left(\|\eta_A \px m_k\|+\|\eta_A m_k\|\right)
+C |\dot \omega| \oo^{-\frac 12}\left(\|\eta_A \px u\|+\|\eta_A u\|\right).
\]
By the expression of~$m_k$, $|x\eta_A|\leq C A$ and~\eqref{eq:modulation}, \eqref{eq:small},
we obtain
\[
\|\eta_A \px m_k\|+\|\eta_A m_k\|
\leq C A \varepsilon \|\rho^2 u\|^2 \leq C A \varepsilon^2 \|\eta_A u\|.
\]
Thus, by the Cauchy-Schwarz inequality, Lemma~\ref{LE:9} and \eqref{eq:modulation},
\begin{align*}
\left| \int \left( 2\Psi_{A,B} \px v_k + \Psi_{A,B}'v_k\right)(\Xe^2 n_k)\right|
& \leq C \|\eta_A\Xe^2 n_k\| \left(\|\eta_A \px v \|+\|\eta_A v \|\right)
\\
& \leq C A \varepsilon^2 \left( \|\eta_A \px u\| +\|\eta_A u\|\right)^2.
\end{align*}

Using Lemmas~\ref{LE:7}, we have
\[
\sum_{k=1,2}\|\eta_A \Xe^2 r_k\|\leq 
C \sum_{k=1,2}\left(\|\eta_A \px q_k\|+\|\eta_A q_k\|\right)
\]
Moreover, by \eqref{eq:small},
\[
|q_1|+|q_2|\leq C |u|^2\leq C \varepsilon |u|
\]
and so
\[
\|\eta_A \px q_k\|+\|\eta_A q_k\|\leq C \varepsilon \left(\|\eta_A \px u\| +\|\eta_A u\|\right)
\]
Thus, as before,
\begin{equation*}
\left| \int \left( 2\Psi_{A,B} \px v_k + \Psi_{A,B}'v_k\right)(\Xe^2 r_k)\right|
\leq C \varepsilon \left(\|\eta_A \px u\| +\|\eta_A u\|\right)^2.
\end{equation*}
Thus, for $\varepsilon$ small (depending on $A$ and $\oo$),
\[
|J_4| \leq \frac CA \left(\|\eta_A\px u\|^2 + \frac{\oo^3}{A^2}\|\eta_A u\|^2\right)
\]
The estimate of~$K_4$ is similar to the one of~$J_4$.

\emph{Estimate of~$J_5$ and $K_5$.} By $|\Phi_B|\leq |x|$, Lemma~\ref{LE:3} and~\eqref{eq:small}, we have
\[
|J_5|+|K_5|\leq C \varepsilon\oo\|\rho^2 z\|^2.
\]
For $\varepsilon$ small, these terms are controlled using~\eqref{eq:std},
\[
|J_5|+|K_5| \leq \frac 1{100} \int \left( |\px z|^2+P_B |z|^2\right).
\]

\emph{Conclusion of the proof of Proposition~\ref{PR:3}.}
Combining the identity~\eqref{111} with the above estimates on the error terms, we have
\begin{align}
\dot{\cJ}+\dot{\cK} 
&\geq \frac 1{100} \int \left(|\px z|^2 + P_B |z|^2\right)\label{222}\\
&\quad - \frac CA \left( \|\eta_A\px u\|^2 + \frac{\oo^3}{A^2}\|\eta_A u\|^2+ \oo\|\rho^2 u\|^2\right).
\nonumber
\end{align}
Moreover, we see that for any $T\geq 0$, using \eqref{333},
\begin{align*}
|\cJ(T)|+|\cK(T)| 
& \leq C \left(\|\Phi_B\|_{L^\infty}+\|\Phi_B'\|_{L^\infty}+\|R_B\|_{L^\infty}\right)(\|\px v\|^2+\|v\|^2)\\
&\leq  C B\oo^{-\frac 32} \varepsilon^2\leq C \varepsilon.
\end{align*}
Therefore, integrating~\eqref{222} on $[0,T]$, we obtain
\begin{multline*}
\int_0^T \int \left(|\px z|^2 + P_B |z|^2\right)
\leq C \varepsilon \\+ \frac CA \int_0^T \left( \|\eta_A\px u\|^2 + \frac{\oo^3}{A^2}\|\eta_A u\|^2
+ \oo\|\rho^2u\|^2\right).
\end{multline*}
Using Lemma~\ref{LE:10}, it follows that
\begin{align*}
\oo^2\int_0^T \|\rho v\|^2 
\leq C\varepsilon + \frac CA \int_0^T \left( \|\eta_A\px u\|^2 + \frac{\oo^3}{A^2}\|\eta_A u\|^2
+ \oo\|\rho^2u\|^2\right),
\end{align*}
which completes the proof of the proposition.
\end{proof}

\subsection{Coercivity property}

Lemma~\ref{LE:9} gives estimates of $v$ in terms of $u$.
The next proposition, inspired by~\cite[Lemma~6]{KMM}, shows conversely that $u$ is controlled by $v$ in the $L^2$ norm with weight $\rho^2$ and some loss. This result uses the orthogonality conditions~\eqref{eq:ortho} on $u$. 

\begin{proposition}\label{PR:4}
There exists~$C>0$ such that, for all $t\geq 0$,
\[
{\oo^2} \|\rho^2 u\| \leq  C\|\rho v\|.
\]
\end{proposition}

The proof of Proposition~\ref{PR:4} follows from the next two lemmas.

\begin{lemma}\label{LE:11}
There exists~$C>0$ such that for any~$\omega\in (0,\frac 18]$ and~$g\in L^2(\R)$,
if 
\[
\langle g, \po \rangle = \langle g, x \po \rangle=0
\]
then
\[
\| \rho^2 g \| \leq C \oo^{-2}\|\rho ( \Xe^2 S^2 \LP g) \|.
\]
\end{lemma}
\begin{proof}
Let~$g$ be as in the statement of the lemma and let~$h = \Xe^2 S^2 \LP g$.
We have
\[
 \px^2 \left( \frac { \LP g}\po\right) = \alpha^2 \frac{h''''}{\po}-2\alpha \frac{h''}{\po} +\frac{h}{\po} .
\]
To simplify notation, we denote by~$f_j$ functions of class~$\cC^\infty$ whose expression may change from line to line, and satisfying
\[
|f_j(x) |\leq C \omega^{- \frac12}e^{\sqrt{\omega}|x|} \mbox{ on~$\R$.}
\]
We also denote~$ \px^{-1} = \int_0^x$ and~$ \px^{-2}= \px^{-1}\cdot \px^{-1}$.
We check
\begin{align*}
\frac{h''}{\po} & = \left(\frac{h}\po\right)'' + \left(f_3{h}\right)'+f_2 h,\\
\frac{h''''}{\po}& =\left(\frac{h}{\po}\right)''''+ \left(f_3{h}\right)'''+\left(f_2h\right)''+\left(f_1 h\right)'+f_0 h.
\end{align*}
Thus,
\[
\px^2 \left( \frac { \LP g}\po\right) = \alpha^2\left(\frac{h}{\po}\right)''''+ \alpha^2\left(f_3{h}\right)'''
+\alpha \left(f_2h\right)'' + \alpha \left(f_1 h\right)'+f_0 h.
\]
By integration and multiplication by~$\po$, we obtain
\begin{equation*}
\LP g = a ( \omega^{\frac 12} x) \po + b\po + \po \sum_{k=-2}^2 \alpha^{\lfloor \frac{k+3}2\rfloor} \px^{k} (f_{k+2} h)
\end{equation*}
where~$a$ and~$b$ are integration constants. We claim that
\[
|a|+|b|\leq C \omega^{-\frac 54} \|\rho h\|.
\]
To prove this, we note that~$\langle \LP g, \Lo\rangle=\langle \LP g, \po' \rangle = 0$,
by~$ \LP \Lo = - \omega \po$,~$ \LP \po' = 0$ and~$\langle g, \po \rangle=0$.
Moreover, $\langle \Lo,x\po\rangle=0$ by parity.
Taking the scalar product of the above expression of~$ \LP g$ by~$\Lo$, we have
\[
|b| \leq \frac C { |\langle \po,\Lo \rangle |} 
\sum_{k=-2}^2 \alpha^{\lfloor \frac{k+3}2 \rfloor} | \langle \po \Lo, \px^{k} ( f_{k+2} h) \rangle |.
\]
We recall from Lemma~\ref{LE:1} that~$|\langle \po,\Lo \rangle |\geq C \omega^\frac12$.
For~$k=0$, we have by the Cauchy-Schwarz inequality,
$| \langle \po \Lo, f_2 h\rangle |\leq C \omega^\frac14 \|\rho h\|$.
For~$k=1,2$, integrating by parts,
\begin{equation*}
 |\langle \po \Lo, \px^{k} (f_{k+2} h)\rangle | = |\langle \px^{k} (\po \Lo), f_{k+2} h\rangle |
 \leq C \omega^{\frac k2+\frac14} \|\rho h\|.
\end{equation*}
For~$k=-1$, we have
\[
\left| \po \Lo \int_0^x f_{1} h\right|
\leq C \omega^{\frac 12}\rho \int_0^x \rho^2 |h| \leq C \omega^{\frac 14}\rho \|\rho h\|,
\]
and so~$| \langle \po \Lo, \px^{-1} ( f_{1} h) \rangle |\leq C \oo^{-\frac 14} \|\rho h\|$.
Last, for~$k=-2$, we have
\[
\left| \po \Lo \int_0^x \int_0^y f_{0} h\right|
\leq C \omega^{\frac 12} \rho \int_0^x \rho \int_0^y \rho^2 |h| \leq C \omega^{-\frac 14} \rho \|\rho h\|,
\]
and so~$| \langle \po \Lo, \px^{-2} (f_{0} h) \rangle |\leq C \omega^{-\frac 34}\|\rho h\|$.
Thus,~$|b|\leq C \omega^{-\frac 54} \|\rho h\|$.
The proof of the estimate for~$a$ is similar.

Recall the notation $I_+$ from \S \ref{S:2.4}.
We have~$g = I_+ [ \LP g] + c \po'$ where~$c$ is a constant. 
In particular, we obtain
\[
g = a I_+ [(\omega^\frac12x)\po] + b I_+ [\po] 
+ \sum_{k=-2}^{2} \alpha^{\lfloor \frac{k+3}2\rfloor} I_+[\po \px^k (f_{k+2} h)] + c \po'.
\]
We estimate each term above for~$x\geq 0$.
We check easily that
\[
|I_+ [(\omega^\frac12x)\po]|+ |I_+ [\po]|\leq C \omega^{-\frac12}
\]
and so~$|a I_+ [(\omega^\frac12x)\po]| + |b I_+ [\po] |\leq C \omega^{-\frac 74}\|\rho h\|$.
Next, for~$k=0$, we have
\[
I_+[\po f_{2} h] = - \po' \int_0^x G\po f_2 h - G\int_x^\infty \po'\po f_2h.
\]
Thus,
\begin{align*}
|I_+[\po f_{2} h]|
&\leq C \omega^{-\frac 12} \ee^{-\sqrt{\omega} x} \int_0^x \ee^{\sqrt{\omega}y}|h|
+ C \omega^{-\frac 12}\ee^{\sqrt{\omega} x} \int_x^\infty \ee^{-\sqrt{\omega}y}|h|\\
&\leq C \omega^{-\frac 12}\rho^{-\frac32}\int_0^\infty \rho^\frac32 |h|
\leq C \omega^{-\frac34}\rho^{-\frac32}\|\rho h\|.
\end{align*}
For~$k=1$, by integration by parts,
\begin{align*}
I_+[\po \px(f_{3} h)]
&= - \po' \int_0^x G\po \px(f_{3} h) - G\int_x^\infty \po'\po \px(f_{3} h)\\
&= \po' \int_0^x (G\po)' f_{3} h + G \int_x^\infty (\po'\po)' f_{3} h+ c_3 \po'
\end{align*}
where~$c_3=G(0)\po(0)f_3(0)h(0)$.
Proceeding as before, we obtain
\[
| I_+[\po \px(f_{3} h)] - c_3 \po'| \leq C \rho^{-\frac32} \omega^{-\frac 14} \|\rho h\|.
\]
For~$k=2$, by integration by parts and using~$\po'' G - \po' G' =1$, we compute
\begin{align*}
I_+[\po \px^2(f_{4} h)]
&= - \po' \int_0^x G\po \px^2(f_{4} h) - G\int_x^\infty \po'\po \px^2(f_{4} h)\\
&= -\po f_4 h - \po' \int_0^x (G\po)'' f_{4} h - G \int_x^\infty (\po'\po)'' f_{4} h + c_4 \po'
\end{align*}
where~$c_4=-G(0)\po(0)(f_4h)'(0)+(G\po)'(0)f_4(0)h(0)$. Proceeding as before, we obtain for~$x\geq 0$,
\[
| I_+[\po \px^2(f_{4} h)] - c_4 \po'| 
\leq C |h| + \rho^{-\frac32}\omega^\frac14\|\rho h\|.
\]
For~$k=-1$, we have
\[
I_+[\po \px^{-1}(f_1 h)] = - \po' \int_0^x G\po \int_0^y (f_1 h) \ud y - G\int_x^\infty \po'\po \int_0^y (f_1h) \ud y.
\]
Thus,
\[
|I_+[\po \px^{-1}(f_1 h)]| 
\leq C \omega^{-\frac12}\rho^{-\frac32}\int_0^\infty\rho^\frac14\int_0^y \rho^{\frac 54}|h| \ud y
\leq C \omega^{-\frac 54}\rho^{-\frac32}\|\rho h\|.
\]
Proceeding similarly, 
\[
|I_+[\po \px^{-2}(f_0 h)]| \leq C \omega^{-\frac 74}\rho^{-\frac32}\|\rho h\|.
\]
We have just proved that
for~$x\geq 0$
\[
|g-\tilde c \po'| \leq C |h|+\omega^{-\frac 74} \rho^{-\frac32}\|\rho h\|
\]
where~$\tilde c=c+c_3+c_4$ .
This estimate also holds for~$x\leq 0$ with the same constant~$\tilde c$.
Taking the scalar product by~$x \po$ and using~$\langle g,x\po\rangle =0$, we obtain
\[
|\tilde c| \leq C \omega^{-2}\|\po\|^{-2} \|\rho h\|\leq C \omega^{-\frac {11}4} \|\rho h\|.
\]
Therefore,
\[
|g| \leq C |g-\tilde c \po'|+C |\tilde c| |\po'|
\leq C |h|+C\omega^{-\frac 74}\rho^{-\frac32} \|\rho h\|.
\]
Multiplying by~$\rho^2$ and taking the~$L^2$ norm, we obtain the result.
\end{proof}

\begin{lemma}
There exists~$C>0$ such that for any~$\omega\in (0,\frac 18]$ and~$g\in L^2(\R)$,
if
\[
\langle g, \Lo \rangle = \langle g, \po' \rangle=0
\]
then
\[
\|\rho^2 g\| \leq C \oo^{-2} \| \rho ( \Xe^2 \MM S^2 g)\|.
\]
\end{lemma}
\begin{proof}
Let~$g$ be as in the statement of the lemma and let~$h = \Xe^2 \MM S^2 g$.
We have
\[
\MM S^2 g = h - 2 \alpha h'' +\alpha^2 h''''
\]
and so using the notation $J_-$ from \S \ref{S:2.4},
\[
 \px^2 \left(\frac g\po\right)
= \frac 1\po\left( J_-\left[h\right] - 2 \alpha J_-\left[h''\right] + \alpha^2 J_- \left[ h''''\right]\right).
\]
By integration, we obtain
\begin{align*}
g 
&= a_1 (\omega^\frac 12 x) \po + b_1 \po+ \po \int_0^x \int_0^y \frac 1\po J_- \left[h\right]\\
& \quad - 2 \alpha \po \int_0^x \int_0^y \frac 1\po J_- \left[ h''\right] 
+ \alpha^2 \po \int_0^x \int_0^y \frac 1\po J_- \left[ h''''\right].
\end{align*}
First, we note that
\[
|J_-[h]|\leq C \omega^{-\frac 12} \rho^{-\frac 32}\int \rho^{\frac 32} |h|
\leq C \omega^{-\frac 34} \rho^{-\frac 32} \|\rho h\|.
\]
For~$x\geq 0$,
\begin{equation*}
 \po \int_0^x \int_0^y \frac 1\po |J_-[h]|
\leq C \omega^{-\frac 34} \|\rho h\| \rho^3\int_0^x \int_0^y \rho^{-\frac32} 
\leq C \omega^{-\frac 74} \rho^{-\frac 32}\|\rho h\|.
\end{equation*}
For the term with~$h''$, we use integration by parts and~$H_1H_2'-H_1'H_2=1$,
\[
J_-[h'']
=-h+ H_1\int_{-\infty}^x H_2'' h + H_2 \int_x^\infty H_1'' h
\]
Thus,
\[
|J_-[h'']|\leq C |h| + C\omega^{\frac 14}\rho^{-\frac 32}\|\rho h\|
\]
and for~$x\geq 0$,
\begin{equation*}
 \po \int_0^x \int_0^y \frac 1\po |J_-[h'']|
\leq C \omega^{-\frac 34} \rho^{-\frac 32}\|\rho h\|.
\end{equation*}
We continue with the term involving~$h''''$, integrating by parts and
using the relation~$H_1''H_2-H_1H_2''=0$,
\[
J_-[h'''']
=h''+(H_1'H_2''-H_1''H_2')h + H_1\int_{-\infty}^x H_2''''h+H_2\int_x^\infty H_1''''h.
\]
For the last three terms in the right hand side, we proceed as before. For the first term, we further compute by integration by parts
\begin{align*}
\po\int_0^x\int_0^y\frac{h''}{\po}
& =a_2 (\omega^\frac 12 x) \po + b_2 \po+ h+ 2\po \int_0^x \frac{h\po'}{\po^2}\\
&\quad + \po \int_0^x \int_0^y h \left(\frac{\omega}{\po}-\frac 13 \po^3\right)
\end{align*}
where we have used~$\po''\po-2(\po')^2=-\omega \po^2 + \frac 13 \po^6$
(from~\eqref{eq:id}) and we proceed as before.
We obtain, for~$a=a_1+a_2$ and~$b=b_1+b_2$, for~$x\geq 0$ 
\[
|g - a (\omega^\frac 12 x) \po - b \po|
\leq C |h| + C\omega^{-\frac 74}\rho^{-\frac32}\|\rho h\|.
\]
This estimate is also true for~$x\leq 0$ with the same constants~$a$ and~$b$.
Using the orthogonality relations~$(g,\Lo)=(g,\po')=0$ to estimate~$a$ and~$b$, we complete the proof as the one
of Lemma~\ref{LE:11}.
\end{proof}

\subsection{End of the proof of Theorem~\ref{TH:1}}\label{s:3.5}
Using first Proposition~\ref{PR:4}, then Proposition~\ref{PR:3} and last Proposition~\ref{PR:2}, we obtain, for all~$T>0$,
\begin{align*}
\oo^6 \int_0^T \|\rho^2 u\|^2 \ud t
& \leq C\oo^2 \int_0^T \|\rho v\|^2\ud t\\
& \leq C\varepsilon + \frac CA \int_0^T\left( \|\eta_A\px u\|^2 + \frac{\oo^3}{A^2} \|\eta_A u\|^2
+ \oo \|\rho^2 u\|^2\right) \uds t\\
& \leq C\varepsilon + \frac {C\oo}A \int_0^T \|\rho^2 u\|^2 \ud t.
\end{align*}
Thus, for~$A$ large enough, independent of $\varepsilon$, but dependent on $\oo$,
\[
\oo\int_0^T \|\rho^2 u\|^2\ud t\leq C \varepsilon \oo^{-5} .
\]
Now, $A$ is fixed to such value.
Using again Proposition~\ref{PR:2}, and passing to the limit $T\to \infty$, we obtain
\[
\int_0^\infty \left( \|\eta_A \px u\|^2  + \oo^3 \|\eta_A u\|^2 
+\oo \|\rho^2 u\|^2\right) \uds t\leq C \varepsilon\oo^{-5}.
\]

From the equation~\eqref{eq:u} of $u$, we compute
\begin{align*}
\frac d{dt} \int |u|^2 \rho^4
&= \int (u_1 (\px u_2) - (\px u_1)u_2) (\rho^4)'
+\int (2\po^2 - 4 \po^4) u_1u_2\rho^4\\
& \quad +\int (\theta_2 u_1 + m_2 u_1 -q_2 u_1 - \theta_1 u_2 -m_1u_2+q_1 u_2)\rho^4.
\end{align*}
Thus, using $|\rho'|\leq C\rho$, $\|u\|_{L^\infty} \leq C$ and \eqref{eq:modulation}, we obtain
\begin{equation}\label{eq:dt}
\left| \frac d{dt} \int |u|^2 \rho^4\right|
\leq C \int \rho^4 \left(|\px u|^2 + |u|^2\right).
\end{equation}
Since $\int_0^\infty \|\rho^2 u\|^2 \ud t<\infty$, there exists a sequence
$t_n\to +\infty$ such that
\[
\lim_{n\to +\infty} \|\rho^2 u(t_n)\| = 0.
\]
Let $t\geq 0$ and $n$ be such that $t_n>t$.
Integrating~\eqref{eq:dt} on $(t,t_n)$, we obtain
\[
\|\rho^2 u(t)\|^2 \leq \|\rho^2 u(t_n)\|^2 + C\int_0^{t_n} \left( \|\rho^2 \px u\|^2 + \|\rho^2 u\|^2\right)\uds t'.
\]
Passing to the limit $n\to +\infty$,
\[
\|\rho^2 u(t)\|^2 \leq C\int_t^{\infty} \left( \|\rho^2 \px u\|^2 + \|\rho^2 u\|^2\right)\uds t'.
\]
Since $\int_0^{\infty} \left( \|\rho^2 \px u\|^2 + \|\rho^2 u\|^2\right)\uds t\leq 
\int_0^\infty \left( \|\eta_A \px u\|^2 + \|\eta_A u\|^2\right)\uds t<\infty$, we have
\[
\lim_{t\to +\infty} \int_t^{\infty} \left( \|\rho^2 \px u\|^2 + \|\rho^2 u\|^2\right)\uds t'=0
\]
and thus
\[
\lim_{t\to +\infty}\|\rho^2 u(t)\|=0.
\]

For any $x,y\in \R$, write
\[
\rho^2(x)|u(t,x)|^2
=\rho^2(y)|u(t,y)|^2 + 
\int_x^y \left[2 \Re\left\{\bar u(t)\px u(t)\right\}\rho^2 + |u(t)|^2(\rho^2)'\right]
\]
so that by the Cauchy-Schwarz inequality,
\[
\rho^2(x)|u(t,x)|^2\leq \rho^2(y)|u(t,y)|^2 + C\|u(t)\|_{H^1(\R)} \|\rho^2 u(t)\|
\]
Integrating for $y\in [0,1]$ and then using \eqref{eq:small}, we obtain
\begin{equation*}
\rho^2(x)|u(t,x)|^2
\leq C\|u(t)\|_{H^1(\R)} \|\rho^2 u(t)\|
\leq C \varepsilon \|\rho^2 u(t)\|.
\end{equation*}
Thus,
\[
\lim_{t\to +\infty}\sup_{x\in \R} \left\{\rho(x)|u(t,x)|\right\}=0.
\]
By \eqref{eq:modulation}, we have
$
|\dot \beta | + | \dot \omega| \leq C \|\rho^2 u\|^2$.
From $\int_0^\infty \|\rho^2 u\|^2\ud t<\infty$, it follows that both $\beta(t)$
and $\omega(t)$ have finite limits as $t\to+\infty$, denoted respectively by~$\beta_+$ and $\omega_+$.
By \eqref{eq:small}, we infer that $|\beta_+|+|\omega_+-\oo|\leq C \varepsilon$.
Last, by~\eqref{def:u} and the triangle inequality, we have
\[
\big|\psi(t,x+\sigma(t)) - \ee^{\ii\gamma(t)} \ee^{\ii \beta_+ x } \phi_{\omega_+}(x)\big|
\leq \big| \ee^{\ii \beta(t) x } \phi_{\omega(t)}(x)
- \ee^{\ii \beta_+ x } \phi_{\omega_+}(x)\big| + |u(t,x)| .
\]
The elementary observation
\[
\lim_{t\to +\infty} \sup_{x\in \R} \big| \ee^{\ii \beta(t) x } \phi_{\omega(t)}(x)
- \ee^{\ii \beta_+ x } \phi_{\omega_+}(x)\big|=0
\]
then completes the proof of Theorem~\ref{TH:1}.

\appendix
\section{}
\begin{lemma}
There exists $B_0\geq 1$ such that for any $B\geq B_0$, the following is true on $\R$, 
\begin{align}
0\leq y \sinh(y) \ee^{\frac1{4B} |y|} & \leq \frac 13 \left( 1+ \frac 1{\sqrt3}\cosh(y) \right)^3,\label{eq:A1} \\
|y| \ee^{\frac1{4B} |y|} & \leq \frac 23 |\sinh(y)|\left(1+ \frac 1{\sqrt3} \cosh(y)\right)\label{eq:A3}.
\end{align}
\end{lemma}
\begin{proof}
First, we check the following inequality
\begin{equation*}
 y\sinh(y) < \frac 13 \left( 1+ \frac 1{\sqrt3}\cosh(y) \right)^3 .
\end{equation*}
For $|y|\leq \frac 54$, we have
$y\sinh(y) \leq \frac 54 \cosh(y)$
and the inequality in this case is a consequence of the fact that for any $a\geq 0$, $15 a < 4 (1+a/\sqrt{3})^3$.

For $|y|\geq \frac 54$, it is easy to check that $|y|\leq \frac 54 |\sinh(y)|$ and so
$y\sinh(y) \le\frac 45 \sinh^2(y)= \frac 45 (\cosh^2(y)-1)$. The inequality is then a consequence of the
fact that for $a\geq 0$, $12(a^2-1)< 5(1+a/\sqrt{3})^3$.

Second, the inequality
\[
1.05 \, |y| \leq \frac 23 |\sinh(y)|\left(1+ \frac 1{\sqrt3} \cosh(y)\right) 
\]
is checked easily since $\cosh(y)\geq 1$, $|\sinh(y)|\geq |y|$ and $\frac 23 (1+ \frac 1{\sqrt3})> 1.05$.
The existence of $B_0$ such that \eqref{eq:A1} and \eqref{eq:A3} hold for $B=B_0$ then follows from standard arguments. Last,~\eqref{eq:A1} and \eqref{eq:A3} are also true for any $B\geq B_0$.
\end{proof}

\end{document}